\documentclass[pt11]{amsart}

 \usepackage{graphicx}

\usepackage{amssymb}
\usepackage{amsmath}

\def\r{\mathbb{R}}

\newtheorem{theorem}{Theorem}[section]
\newtheorem{corollary}[theorem]{Corollary}
\newtheorem{definition}[theorem]{Definition}
\newtheorem{lemma}[theorem]{Lemma}
\newtheorem{remark}[theorem]{Remark}
\newtheorem{proposition}[theorem]{Proposition}

\usepackage{url}

\title{The translating soliton equation}

\author{Rafael L\'opez}
\thanks{Partially supported by the grant no. MTM2017-89677-P, MINECO/AEI/FEDER, UE.}
\address{Departamento de Geometr\'{\i}a y Topolog\'{\i}a\\ Instituto de Matem\'aticas (IEMath-GR)\\
 Universidad de Granada\\
 18071 Granada, Spain}
\email{rcamino@ugr.es}

\begin{document}
%

\begin{abstract}
We give an analytic approach to the translating soliton equation with a special emphasis in the   study of the Dirichlet problem in  convex domains of the plane.
\end{abstract}

\maketitle              

%

\section{Historical introduction and motivation}

In this paper we consider the   equation of mean curvature type
\begin{equation}\label{eqs}
\mbox{div}\left(\frac{Du}{\sqrt{1+|Du|^2}}\right)=\frac{1}{\sqrt{1+|Du|^2}}
\end{equation}
in  a smooth domain $\Omega\subset\r^2$, where  $u\in C^2(\Omega)\cap C^0(\overline{\Omega})$. We   call (\ref{eqs})   the {\it translating soliton equation}. The geometry behind this equation is the following. Let $(x,y,z)$ be the canonical coordinates in  Euclidean space $(\r^3,\langle,\rangle)$ and denote $\Sigma_u=\{(x,y,u(x,y)):  (x,y)\in\Omega\}$ the graph of a function $u$.  The left-hand side of (\ref{eqs}) is twice the mean curvature $H$ of $\Sigma_u$ at each point $(x,y,u(x,y))$. Here $H$ is the average of the principal curvatures    calculated with respect to    the unit normal vector field 
$$N=\frac{1}{\sqrt{1+|Du|^2}}(-Du,1).$$
Hence   the right-hand side of (\ref{eqs}) is the $z$-coordinates of $N$. 
Consequently, a surface  in Euclidean space satisfies locally the translating soliton equation if and only if the mean curvature at each point   is the half of the cosine of the angle that makes $N$ with the vertical direction $\vec{a}=(0,0,1)$.   

As far as the author knows, it was S. Bernstein in 1910 the first author that studied equation (\ref{eqs}) in a couple of papers \cite{be1,be2} in the context of the solvability of the Dirichlet problem of   elliptic equations. In \cite[p. 240]{be1}, the translating soliton equation  appears numbering as (6) and Bernstein names    {\it l'\'equation des surfaces, dont la courbure en chaque point is proportionnelle (\'egale) au cosinus de l'angle de la normale en ce point avec l'axe des $z$}. On the other hand, in \cite[p. 515]{be2}  Bernstein considers     a family of equations numbered as (2')   in classical notation  
\begin{equation}\label{eqs2}
 (1+q^2)r-2pqs+(1+p^2)t=(1+p^2+q^2)^{n/2},
 \end{equation}
where $n$ is an integer number. In particular, for $n=2$ this equation coincides with (\ref{eqs}). The Dirichlet problem consists into find a smooth solution of (\ref{eqs2}) with boundary data
\begin{equation}\label{eqsb}
u=\varphi \quad \mbox{on $\partial\Omega$},
\end{equation}
where $\varphi\in C^0(\partial\Omega)$.  Bernstein proved that (\ref{eqs2})-(\ref{eqsb}) is solvable for arbitrary analytic functions $\varphi$ when $\Omega$ is an analytic convex  domain and  $n\leq 2$. In particular, this result holds for the translating soliton equation.

Sixty years later, the second approach to equation (\ref{eqs}) is due to J. Serrin. In the eighty-pages article \cite{se}, Serrin gave   a systematic treatment of the Dirichlet problem for a large  class of  quasilinear non-uniformly second order elliptic equations. Following the Leray-Schauder fixed point theorem and H\"{o}lder estimates theory of Ladyzenskaja and Ural'ceva, Serrin establishes the necessary and sufficient conditions for the solvability of the Dirichlet problem for arbitrary boundary data. Possibly, the most known result of \cite{se} is  the case of the  constant mean curvature equation, that is, when the right-hand side of (\ref{eqs}) is replaced by a constant $2H$. In such a case,  the Dirichlet problem has a   solution for arbitrary smooth boundary data $\varphi$ if and only if   the   curvature $\kappa$ of $\partial\Omega$ with respect to the inward normal direction satisfies     $\kappa\geq 2|H|$. If the solution exists, it is unique. See   \cite{se2} when the boundary $\partial\Omega$ is not necessarily smooth. 

However, the article \cite{se} covers many other types of quasilinear elliptic equations and this is the situation of the translating soliton equation.  Exactly in pages 477--478, Serrin considers two families of quasilinear elliptic equations and one of them coincides with (\ref{eqs2}). The   equation (96) of \cite{se} is
\begin{equation}\label{eqs3}
 (1+q^2)r-2pqs+(1+p^2)t=2H(1+p^2+q^2)^{n/2},
 \end{equation}
where now $H$ and $n$ are two real constants: recall that in  (\ref{eqs2}),      $n$ is     an integer number. Notice that  if  $n=3$, the expression (\ref{eqs3}) is the constant mean curvature equation. As a consequence of the results previously obtained,  Serrin proves  the following existence result (\cite[p. 478]{se}).

\begin{theorem}\label{t1} Let $\Omega\subset\r^2$ be a bounded   $C^2$-domain. Then (\ref{eqs2})-(\ref{eqsb}) is solvable for arbitrarily given $C^2$ function $\varphi$ 
\begin{enumerate}
\item when $n\leq 2$, if and only if $\kappa\geq 0$, and  
\item when $2<n<3$, if and only if $\kappa>0$. 
\end{enumerate} 
When $n>3$, the Dirichlet problem is not generally solvable, whatever the domain.
\end{theorem}

Definitively, for the translating soliton equation, we conclude:
\begin{corollary}\label{c1}
 Let $\Omega\subset\r^2$ be a bounded   $C^2$-domain. Then (\ref{eqs})-(\ref{eqsb}) is solvable for arbitrarily given $C^2$ function $\varphi$   if and only if the inward curvature satisfies $\kappa\geq 0$.
\end{corollary}

\begin{remark}
\begin{enumerate}
\item The case $n=3$ in (\ref{eqs3}), which does not appear in Theorem \ref{t1}, is the constant mean curvature equation, where  the solvability occurs if and only if  $\kappa\geq 2|H|$.
\item Serrin generalizes the result of Bernstein in \cite{be1} changing analyticity  by smoothness of $\Omega$.
\item The results of \cite{se}   for the equation (\ref{eqsb}) are established in arbitrary dimension.  
\end{enumerate}
\end{remark}

Possibly due to the lengthy  paper \cite{se}, equation (\ref{eqs}) seemed to be forgotten in the literature.  It is in 80's when the translating soliton equation appears  in two different contexts at the same time. 

Firstly in the singularity theory of the mean curvature flow  of Huisken and Ilmanen \cite{hsi1,il}. A {\it translating soliton} is a surface $\Sigma\subset\r^3$ that is  a solution of the mean curvature flow when $\Sigma$ evolves purely by translations along some direction $\vec{a}\in\r^3\setminus\{0\}$. In other words, $\Sigma$ is a translating soliton  if    $\Sigma+t\vec{a}$, $t\in \r$, satisfies that fixed $t$,   the normal component of the velocity vector $\vec{a}$   at each point is equal to the mean curvature at that point. For the initial surface $\Sigma$, this implies that $2H=\langle N,\vec{a}\rangle$.  After a change of coordinates, if $\vec{a}=(0,0,1)$, then    $2H=\langle N,\vec{a}\rangle$ coincides  locally with (\ref{eqs}).  Translating solitons appear    in the singularity theory of the mean curvature flow. After scaling, near type II-singularity points on the surfaces evolved by mean curvature vector,  Huisken, Sinestrari and White demonstrated  that the limit flow with initial convex surface is a convex translating soliton (\cite{hsi1,hsi2,wh0}). On the other hand,   Ilmanen observed that $\Sigma$ translates with velocity $\vec{a}$ if and only if it is stationary for the weighted area $\int_\Sigma e^{\langle p,\vec{a}\rangle}dA$. 
 In fact, $2H=\langle N,\vec{a}\rangle$ is the Euler-Lagrange equation for this functional and thus $\Sigma$ is a minimal surface with respect to the Riemannian metric $e^{\langle p,\vec{a}\rangle}\langle,\rangle$.

 From the last viewpoint, equation (\ref{eqs}) links with the   theory of manifolds with density of   Gromov (\cite{gr}).  Exactly, let   $e^\phi$ be a   positive density function in $\r^3$, $\phi\in C^\infty(\r^3)$, which serves as a weight for  the volume and the surface area. Note that this is not equivalent to scaling the metric conformally by $e^\phi$ because the area and the volume change with different scaling factors.  For a given compactly supported variation of $\Sigma_t$ of $\Sigma$ that fixes the boundary $\partial\Sigma$ of $\Sigma$, let    $A_\phi(t)$ and $V_\phi(t)$ denote the weighted area and the enclosed weighted volume of $\Sigma_t$, respectively. Then the first variations of $A_\phi(t)$ and $V_\phi(t)$ are  
$$A'_\phi(0)=-2\int_\Sigma H_\phi \langle N,\xi\rangle\  dA_\phi,\quad V_\phi'(0)=\int_\Sigma \langle N,\xi\rangle\ dA_\phi,$$
where $\xi$ is the   variational vector field of $\Sigma_t$ and $H_\phi=H-\langle N,\nabla\phi\rangle/2$ is called the {\it weighted mean curvature}. If we choose   $\phi(p)=\langle p,\vec{a}\rangle$, $p\in\r^3$, then 
\begin{equation}\label{hf}
H_\phi=H-\frac{\langle N,\vec{a}\rangle}{2}.
\end{equation}
   We say that $\vec{a}$ is the {\it density vector}. Thus we have the next characterizations of a translating soliton.

\begin{proposition} Let $\Sigma$ be a surface in $\r^3$. The following statements are equivalent:
\begin{enumerate}
\item $\Sigma$ satisfies locally (\ref{eqs}).
\item $\Sigma$ translates with velocity $\vec{a}$ by means of the mean curvature flow.
\item $\Sigma$ is a critical point of the area $A_\phi$ for the density $\phi(p)=\langle p,\vec{a}\rangle$
\end{enumerate}
\end{proposition}
  
In view of both approaches, we point out that    similar results   of Theorem \ref{t1} and Corollary \ref{c1} have been recently treated in   the literature. We  indicate some of them.

\begin{enumerate}
\item Corollary \ref{c1} appears in \cite[Th. 2]{ber} assuming $\Omega$ is contained in a disc of radius $1$ and satisfying an enclosing sphere condition. But in a Remark, Bergner asserts  that the assumption to be contained in a ball can drop if there exist $C^0$ estimates, such as it occurs for  (\ref{eqs}): see Proposition \ref{pr-31} below.
\item Corollary \ref{c1}    appears in \cite[Rem. 3]{ma}. Initially, it is assumed that  $|\Omega|<4\pi$ in a general result, but for (\ref{eqs}) this hypothesis drops. Using the same proof than in   \cite{ma}, the existence holds for $n\leq 2$ in equation (\ref{eqs2}) under the assumption that $|\Omega|<4\pi$.
\item Theorem \ref{t1} appears in   \cite[Lem. 2.2]{jj} for  $0<n<3$ assuming $\kappa>0$ and $|\Omega|<4\pi$.
\item  Corollary \ref{c1} appears in \cite[Th. 1.1]{sx} assuming that     $\mbox{diam}(\Omega)<2$.
\end{enumerate}

We finish this section giving two generalizations of the translating soliton equation. First,   consider the flow of surfaces by powers of mean curvature according to \cite{sh1,sh2,sw}. If $\alpha>0$ is a constant, then the surface $z=u(x,y)$ evolves by translations of the $H^\alpha$-power of mean curvature flow if
$$\mbox{div}\left(\frac{Du}{\sqrt{1+|Du|^2}}\right)=\left(\frac{1}{\sqrt{1+|Du|^2}}\right)^\alpha.$$
 Notice that  this equation coincides with (\ref{eqs2}) of Bernstein and Serrin with the relation $n=3-\alpha$. 

The second generalization is by considering   critical points of the  area $A_\phi$ for a  fixed weighted volume. As a consequence of the Lagrange multipliers, $\Sigma$ satisfies that  $H_\phi$ is a constant function and thus, in nonparametric form, we have
\begin{equation}\label{eq21}
\mbox{div} \frac{Du}{\sqrt{1+|Du|^2}} = \frac{1}{\sqrt{1+|Du|^2}}+\mu,
\end{equation}
where $\mu$ is a constant. This equation   has received a recent interest: \cite{bgm,elo,lo2}. Even more general, we may study the mean curvature flow with a forcing term, so the constant $\mu$ in (\ref{eq21}) is replaced by a function $f=f(u,Du)$ (\cite{jl0,ss}). For example, the mean curvature type equation 
$$\mbox{div} \frac{Du}{\sqrt{1+|Du|^2}} = H_1(x,u,Du)+H_2(x,u,Du)\frac{1}{\sqrt{1+|Du|^2}}$$
has been studied in \cite{ber,jl,ma}. 

{\bf Convention.}  After a change of coordinates, we will assume that $\vec{a}=(0,0,1)$.

 This paper is organized as follows. In Section \ref{sec2} we recall the translating solitons that are invariant by a uniparametric group of   translations and of rotations.   Section \ref{sec3} is devoted to the tangency principle and some consequences derived by its applications to control the shape of a compact translating soliton. Sections \ref{sec4} and \ref{sec5} solve the Dirichlet problem on bounded convex domains for the translating soliton equation (\ref{eqs}) and the constant weighted mean curvature equation (\ref{eq21}), respectively. Here the boundary gradient estimates are obtained by means of the classical maximum principle to suitable  choices
of barrier functions.   Finally, in Section \ref{sec6}  we  study the Dirichlet problem for (\ref{eqs}) in unbounded domains. We will  consider two cases, namely,   the domain is a strip and the boundary data are two copies of a convex function or     the domain is an unbounded convex domain contained in a strip and the boundary data are constant.

\section{Examples of translating solitons}\label{sec2}

Let  $\{e_1,e_2,e_3\}=\{(1,0,0),(0,1,0),(0,0,1)\}$ denote the canonical basis of $\r^3$. The Euclidean plane $\r^2$ will identified with plane of equation $z=0$.  We also use the terminology horizontal and vertical to indicate an orthogonal direction  to $\vec{a}$ or a parallel direction to $\vec{a}$, respectively.

Notice that any translation of  $\r^3$  preserves  solutions of the  translating soliton equation.  The same  occurs for     a rotation about an axis parallel to $e_3$.  Also,   equation (\ref{eqs})  is preserved by reversing the orientation on the surface.

In this section, we   are interested by examples of translating solitons that are invariant by a uniparametric group of motions, more precisely,   surfaces invariant along one direction and surfaces of revolution. In both cases, the equation (\ref{eqs}) converts into an ODE and one may apply the standard theory.  

\subsection{Cylindrical surfaces}

A  cylindrical surface $\Sigma$ is a surface   invariant along a direction $\vec{v}\in\r^3$, or in other words,   $\Sigma$ is a  ruled surface where all the rulings are parallel to $\vec{v}$.  We ask for those translating solitons of cylindrical type. Notice that there is not an {\it a priori} relation between the direction $\vec{v}$ and the density vector $\vec{a}$.

A parametrization of $\Sigma$ is $X(t,s)=\gamma(s)+t\vec{v}$, $t\in\r$ and $\gamma:I\subset\r\rightarrow\r^3$ is a planar curve orthogonal to $\vec{v}$.  We parametrize $\gamma=\gamma(s)$ by the arc length  $s$  such that   $\gamma'(s)\times {\bf n}(s)=\vec{v}$, being ${\bf n}$   the unit principal normal vector of $\gamma$.  The Gauss map of $\Sigma$ is $N(X(s,t))={\bf n}(s)$ and $2H=\kappa(s)$, being $\kappa$   the inward curvature of $\gamma$ as a planar curve. Thus $\Sigma$ is a translating soliton if   $\kappa(s)=\langle{\bf n}(s),\vec{a}\rangle$, hence we conclude that there is not a relation between the vectors $\vec{v}$ and $\vec{a}$. For example, if $\vec{v}$ is parallel to $\vec{a}$, then $\langle {\bf n}(s),\vec{a}\rangle=0$ for every $s\in I$, so $\gamma$ is a straight line and $\Sigma$ is a plane parallel to $\vec{a}$.

In a first step, we investigate the case that  $\vec{v}$ is orthogonal to $\vec{a}$ which, after a rotation about $\vec{a}$, we suppose $\vec{v}=e_1$. Let us observe that a vertical plane parallel to $e_1$, that is, a plane parallel to the $yz$-plane,  is a translating soliton of  cylindrical type. If we write $\gamma$ as   $z=w(y)$, then (\ref{eqs}) converts to
$$
w''=1+w'^2.
$$
By simple quadratures, the solution of this equation is 
\begin{equation}\label{gr}
w(y)=-\log(\cos(y+b))+a,\quad a,b\in\r,
\end{equation}
and this solution  is called the {\it grim reaper}. Although this holds for graphs $z=w(y)$, it is true in general:   if there is a vertical tangent vector at some point of $\gamma$, then $\gamma$ is a vertical line by uniqueness of ODE. This can be also obtained as follows. We parametrize $\gamma$ by the arc length. Then $\gamma(s)=(0,y(s),z(s))$, with $y'(s)=\cos\psi(s)$, $z'(s)=\sin\psi(s)$ for some function $\psi$. Then  $X(t,s)=(t,y(s),z(s))$  and (\ref{eqs}) becomes $2\psi'(s)=\cos\psi(s)$. If $\gamma$ is not a graph on the $y$-axis,  there is $s=s_0$ such that $\cos\theta(s_0)=0$. By uniqueness, the solution  is $\theta(s)=\pm\pi/2$,  $\gamma(s)=(0,a,\pm s+b)$, $a,b\in\r$, $\gamma$ is a vertical line and $\Sigma$ is the vertical plane of equation $y=a$.

Once obtained the translating solitons of cylindrical type when the vector $\vec{v}$ is orthogonal to $\vec{a}$, the rest of cylindrical surfaces are obtained by rotating the about surfaces about   a horizontal axis. The resulting surfaces are all  translating solitons of cylindrical type (after translations and rotations about a vertical axis). We present these surfaces, which will be called {\it grim reapers} again (Figure \ref{fig1}).

\begin{definition}\label{d-22}
 The uniparametric family of  grim reapers $w_\theta=w_\theta(x,y)$ are defined as 
\begin{equation}\label{ws}
w_\theta(x,y)=-\frac{1}{(\cos\theta)^2}\log(\cos(\cos\theta  y))+(\tan\theta)x+a,
\end{equation}
where    $\theta\in (-\pi/2,\pi/2)$, $a\in\r$. 
\end{definition}

Here we recall that planes parallel to the $xz$-plane are cylindrical translating solitons, which would correspond with the critical values $\theta=\pm\pi/2$.

\begin{proposition} All translating solitons of cylindrical type are planes parallel to the $xz$-plane or the grim reapers $w_\theta$.
\end{proposition}
\begin{proof} 
If $\vec{v}=\vec{a}$, we know that the surface is (\ref{gr}), which coincides, up to a reparametrization, with (\ref{ws}) for the choice $\theta=0$. 

 Suppose $\vec{v}$ be a vector which is not orthogonal to $\vec{a}$. After a rotation with respect to the $z$-axis, we assume that $\vec{v}=\cos\theta e_1+\sin\theta e_3$, $|\theta|<\pi/2$. Let $\vec{e}=-\sin\theta e_1+\cos\theta e_3$. We write the generating curve as a graph $g=g(s)$ on the $y$-axis. Then  parametrization of the surface is $X(t,s)=s e_2+g(s)\vec{e}+t\vec{v}$. A  computation shows that (\ref{eqs}) writes as $g''=\cos\theta(1+g'^2)$ and its integration gives $g(s)=-\log(\cos(\cos\theta s+b))/\cos\theta+a$, $a,b\in\r$. Then
$$X(t,s)=(-\sin\theta g(s)+t\cos\theta,s,t\sin\theta+\cos\theta g(s)).$$
Writing $X(t,s)=(x,y,u(x,y))$,   we deduce easily that $u$ coincides with the function $w_\theta$ in (\ref{ws}). 
\end{proof}

The maximal domain of  $w_\theta$ is     the strip 
$$\Omega^{\theta}=\left\{(x,y)\in\r^2: -\frac{\pi}{2\cos\theta}<y<\frac{\pi}{2\cos\theta}\right\}.$$
 In particular, if $0\leq \theta_1<\theta_2$, it follows that  $\Omega^{\theta_1}\subset\Omega^{\theta_2}$ and thus the domain $\Omega^0$, namely,  
 \begin{equation}\label{eq-do}
 \Omega^0=\{(x,y)\in\r^2:-\frac{\pi}{2}<y<\frac{\pi}{2}\},
 \end{equation}
 is contained in all  $\Omega^{\theta}$ for any $\theta\in (-\pi/2,\pi/2)$.

  \begin{figure}[h]
\begin{center}
\includegraphics[width=.4\textwidth]{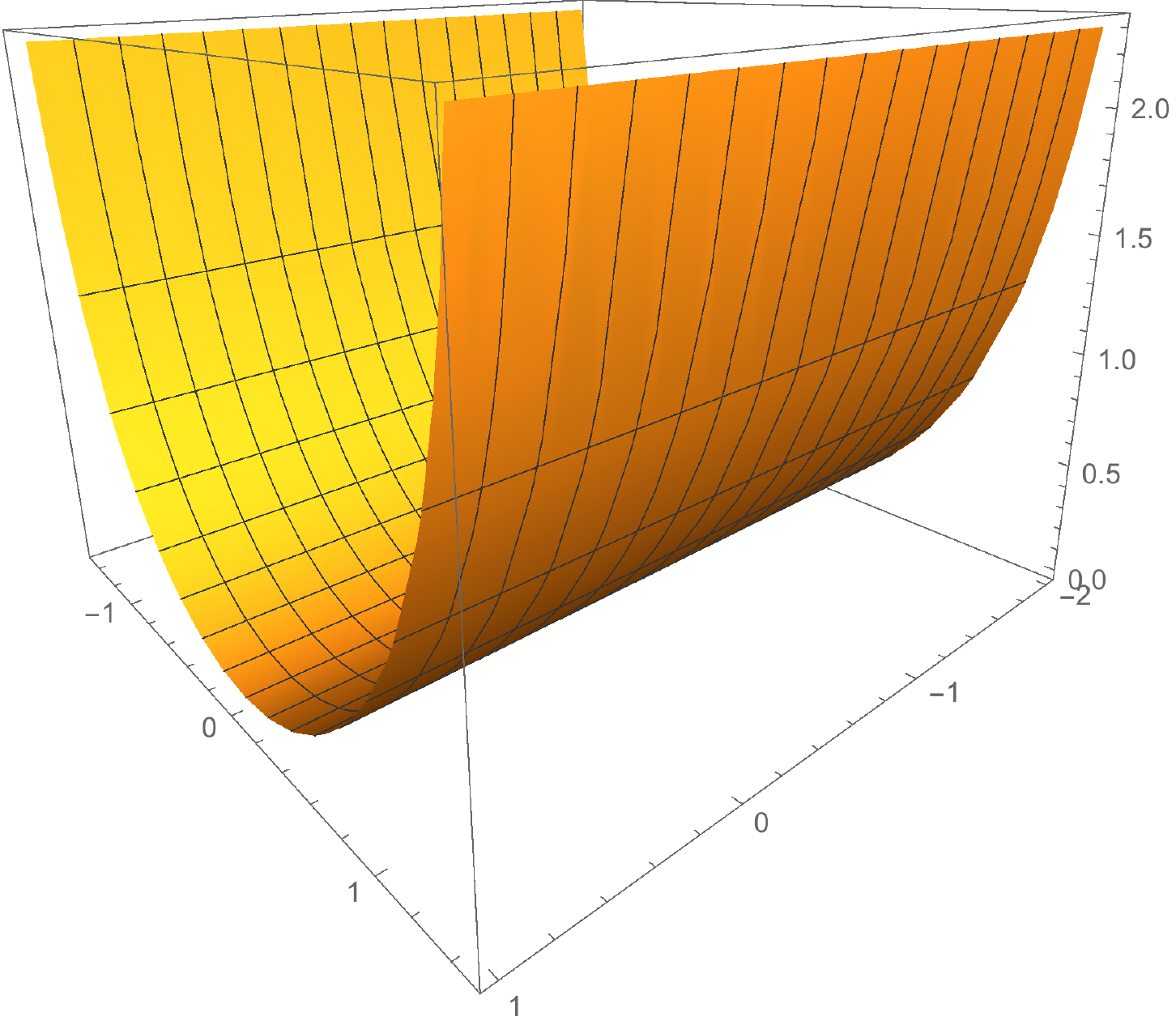}\quad \includegraphics[width=.35\textwidth]{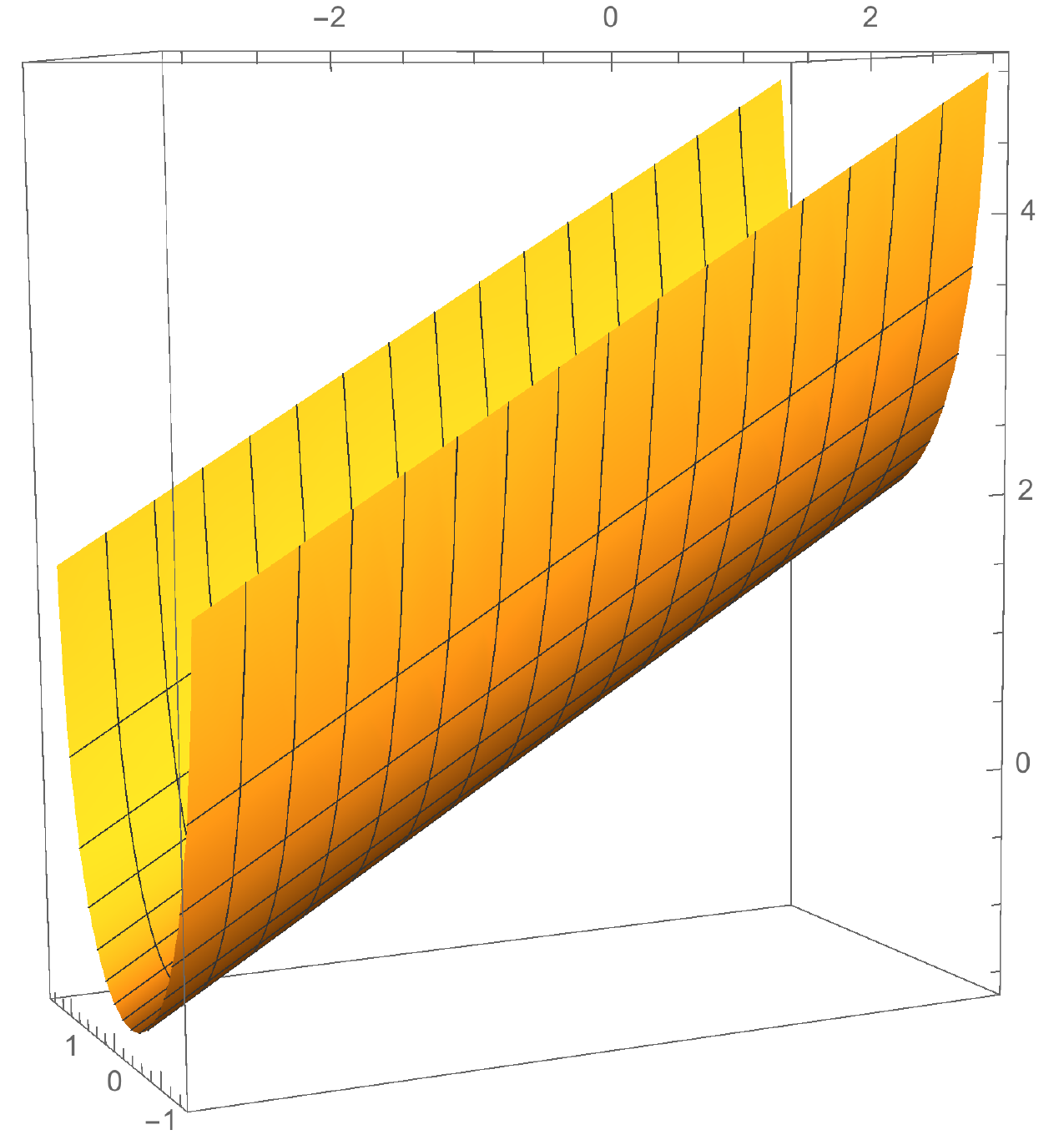}
\end{center}
\caption{The grim reapers $w_\theta$. Left: $\theta=0$; Right: $\theta=\pi/6$}\label{fig1}
\end{figure}

 \subsection{Rotational surfaces}
The second family of translating solitons of our interest are those ones of rotational type. If $\Sigma$ is a surface of revolution about a rotation axis   $\vec{v}$,   we ask about the relation between the vector $\vec{v}$ and the density vector $\vec{a}$. 

\begin{proposition} Let $\Sigma$ be a surface of revolution with respect to the vector $\vec{v}$. If $\Sigma$ is a translating soliton, then $\vec{v}$ is parallel to $\vec{a}$ or $\Sigma$ is a plane parallel to $\vec{a}$  where  $\vec{v}$ is  orthogonal to $\vec{a}$.
\end{proposition}

\begin{proof}
The value of the mean curvature $H$ is constant along  a parallel of $\Sigma$. On the other hand,  the Gauss map $N$ makes a constant angle with $\vec{v}$ along a parallel of the surface. Since $2H=\langle N,\vec{a}\rangle$,  the function  $\langle N,\vec{a}\rangle$ is constant along every parallel of $\Sigma$. 
Hence, we have only two possibilities, namely, $\vec{v}$ is parallel to $\vec{a}$ or $\langle\vec{v},\vec{a}\rangle=0$ with   $\langle N,\vec{a}\rangle=0$ on $\Sigma$. In the latter case,   $\Sigma$ is a plane parallel to $\vec{a}$.
\end{proof}

After  a translation of $\r^3$, we will assume that the rotation axis is the $z$-axis. If  we parametrize $\Sigma$ as $z=u(r)$, $r^2=x^2+y^2$, equation (\ref{eqs}) becomes
\begin{equation}\label{ura}
u''+\frac{u(1+u'^2)}{r}=1+u'^2.
\end{equation}
Therefore, by standard theory of ODE, there are solutions of (\ref{ura}) of initial conditions $u(r_0)=u_0$, $u'(r_0)=u_0'$, with $r_0>0$. The classification of the translating solitons of rotational type was done in \cite{aw,css}:   see Figure  \ref{fig2}.   

\begin{definition} There are two types of rotational translating solitons depending if the surface meets or does  not meet the rotation axis:
\begin{enumerate}
\item Bowl solitons. They are strictly convex entire graphs with a global minimum in the $z$-axis and intersect orthogonally the rotation axis. The surfaces are asymptotic to a paraboloid. 
\item Surfaces of winglike shape. These surfaces do not intersect the rotation axis.
\end{enumerate}
\end{definition}

  \begin{figure}[hbtp]
\begin{center}
\includegraphics[width=.4\textwidth]{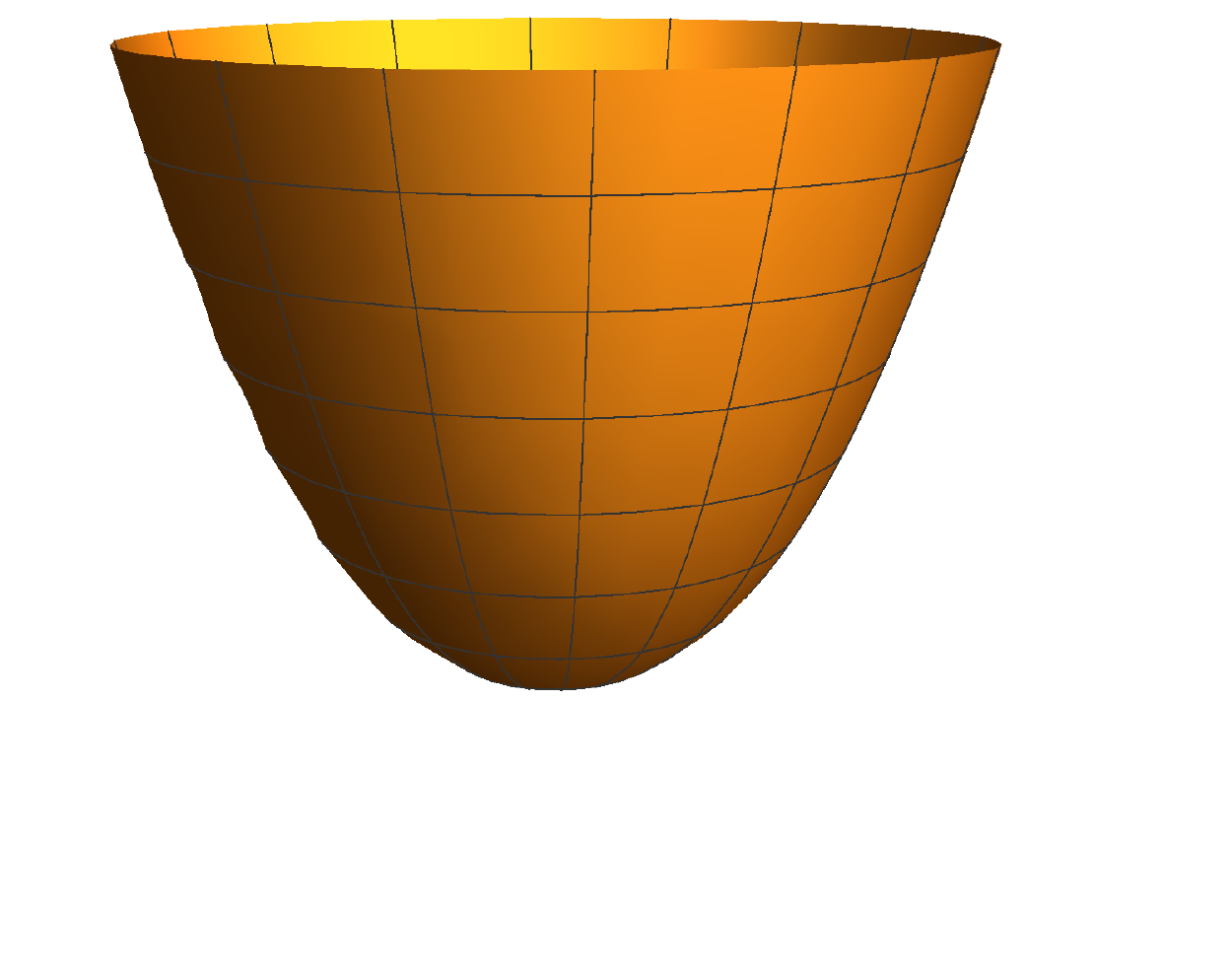}\quad \includegraphics[width=.4\textwidth]{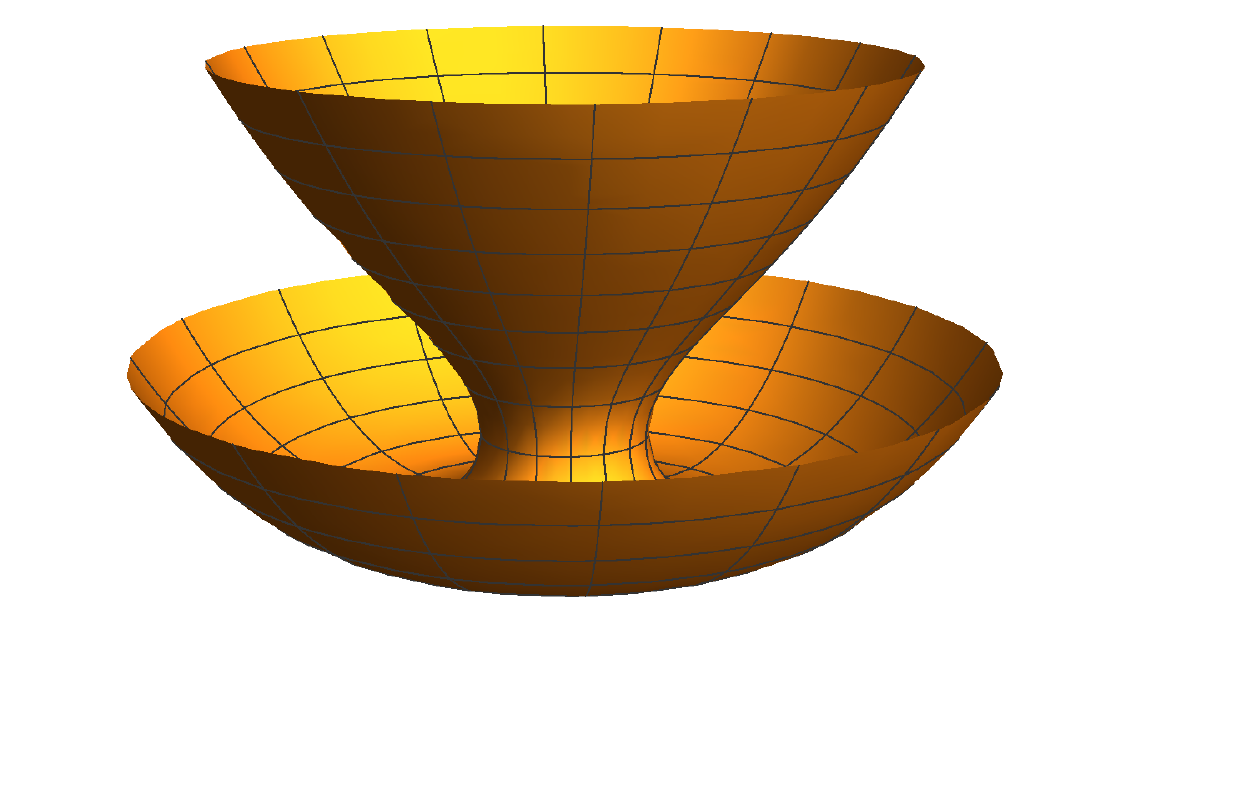}
\end{center}
\caption{Rotational translating solitons.   Left: the bowl soliton; Right: surface with winglike-shape}\label{fig2}
\end{figure}

The bowl soliton   corresponds with the solution of (\ref{ura}) with initial condition  $u(0)=u'(0)=0$ where the existence   is  not a direct   consequence of the standard theory because (\ref{ura}) presents a singularity at $r=0$.   On the other hand, the winglike-shape solutions corresponds with solutions of (\ref{ura}) with $r_0>0$ and $u'(r_0)=0$,  whose existence is immediate.

The existence of the bowl soliton was done in  \cite[Cor. 3.3]{aw}. The authors solve   (\ref{eqs}) in  a round disk with Neumann boundary condition $\partial u/\partial n=\cos\alpha/\sqrt{1+|Du|^2}$ and, after an argument of continuity varying the parameter  $\alpha$, they obtain the desired rotational solution. In this paper, we give two alternative proofs of the existence of the bowl solitons. One will appear in Remark \ref{rem1} using Corollary \ref{c1} and an argument by means of the Alexandrov reflection method. We  now present the other proof, which follows standard techniques of radial solutions for some   equations of mean curvature type  (\cite{be,cco}).    We write (\ref{ura}) as 
\begin{equation}\label{rot1}
\frac{u''(r)}{(1+u'(r)^2)^{3/2}}+\frac{u'(r)}{r\sqrt{1+u'(r)^2}}= \frac{1}{\sqrt{1+u'(r)^2}}.
\end{equation}
 Multiplying   (\ref{rot1}) by $r$, and integration by parts, we wish to establish the existence of a classical solution of
\begin{equation}\label{rot-r}
\left\{\begin{array}{ll}
 \left({\displaystyle \frac{r u'(r)}{\sqrt{1+u'(r)^2}}}\right)'= {\displaystyle \frac{r}{\sqrt{1+u'(r)^2}}},&\mbox{ in $(0,\delta)$}\\
u(0)=0, \quad u'(0)=0.&
\end{array}\right.
\end{equation}
Let us observe that    equation     (\ref{rot-r}) is singular at $r=0$.

\begin{proposition}\label{pr-exi}
The initial value problem (\ref{rot-r})  has a solution $u\in C^2([0,R])$ for some $R>0$ which depends continuously on the initial data.
\end{proposition}

\begin{proof}
Define the functions $g:\r\times\r\rightarrow\r$ and $f:\r\rightarrow\r$ by
$$g(x,y)=\frac{1}{\sqrt{1+y^2}},\ f(y)=\frac{y}{\sqrt{1+y^2}}.$$
It is clear that a function $u\in C^2([0,\delta])$, for some $\delta>0$, is a solution of (\ref{rot-r}) if and only if   $(rf(u'))'=r g(u,u')$ and $u(0)=0$, $u'(0)=0$.

Fix $\delta>0$ to be determined later, and define the operator ${\mathsf T}$ by
$$({\mathsf T}u)(r)=a+\int_0^rf^{-1}\left(\int_0^s\frac{t}{s} g(u') dt\right)ds.$$
Note that a fixed point of the operator ${\mathsf T}$ is a solution of the initial value problem (\ref{rot-r}). We claim now that ${\mathsf T}$ is a contraction in the space $C^1([0,\delta])$ endowed  with the usual norm $\|u\|=\|u\|_\infty+\|u'\|_\infty$. To see this, the functions $g$ and  $f^{-1}$ are  Lipschitz continuous of constant $L>0$ in
$[-\epsilon,\epsilon]\times[-\epsilon,\epsilon]$ and $[-\epsilon,\epsilon]$ respectively, provided $\epsilon<1$. Then for all $u,v\in\overline{B(0,\epsilon)}$ and for all $r\in [0,\delta]$,
$$|({\mathsf T}u)(r)-({\mathsf T}v)(r)|\leq\frac{L^2}{4} r^2\left( \|u-v\|_\infty+\|u'-v'\|_\infty\right)$$
$$|({\mathsf T}u)'(r)-({\mathsf T}v)'(r)|\leq\frac{L^2}{2} r\left( \|u-v\|_\infty+\|u'-v'\|_\infty\right)$$
By choosing $\delta>0$ small enough, we deduce that ${\mathsf T}$ is a contraction in the closed ball $\overline{B(0,\delta)}$ in $C^1([0,\delta])$. Thus the Schauder Point Fixed theorem proves the existence of one fixed point of $\mathsf{T}$, so  the existence of  a local solution of the initial value problem  (\ref{rot-r}). This solution lies in $C^1([0,\delta])\cap C^2((0,\delta])$. The $C^2$-regularity up to $0$ is verified directly by  using the L'H\^{o}pital rule because (\ref{rot1})  leads to
     $$u''(0)+\lim_{r\rightarrow 0}\frac{u'(r)}{r}=1,$$
that is,
$$\lim_{r\rightarrow 0} u''(r)=\frac12.$$
The continuous dependence of local solutions on the initial data  is a consequence of the continuous dependence of the fixed points of ${\mathsf T}$.
 \end{proof}

From the   classification of the rotational translating solitons, we observe    that there do not exist closed surfaces (compact without boundary).  Even more, we prove that there are not   closed translating solitons. Usually the proof that appears in the literature of this result uses the touching principle (see Proposition  \ref{p-tan} below). However, it is easier the following argument that we present, which only utilizes  the divergence theorem (\cite{lo2}).

\begin{proposition} \label{pr-closed}
There do not exist closed translating solitons.
\end{proposition}

\begin{proof} The proof is by contradiction. Suppose that $\Sigma$ is a closed translating soliton. Since    the Laplacian $\Delta$ of the height function $\langle p,\vec{a}\rangle$ is $\Delta\langle p,\vec{a}\rangle=2H\langle N,\vec{a}\rangle$, and $2H=\langle N,\vec{a}\rangle$, then 
$$
\Delta\langle p,\vec{a}\rangle=\langle N,\vec{a}\rangle^2.
$$
 Integrating in $\Sigma$ and   using the divergence theorem, we deduce  
\begin{equation}\label{com}
0= \int_\Sigma\langle N,\vec{a}\rangle^2\ d\Sigma,
\end{equation}
because $\partial\Sigma=\emptyset$. Hence $\langle N,\vec{a}\rangle=0$ in $\Sigma$. This is a contradiction because on a closed surface, the Gauss map $N$ is   surjective   on the unit sphere $\mathbb{S}^2$. 
\end{proof}

 \section{Properties of the solutions of the translating soliton equation}\label{sec3}

This section establishes some properties of the solutions $u$ of   the translating soliton equation, with a special interest in the control of $|u|$ and $|Du|$ when $\Omega$ is a bounded domain.   

It is easily seen  that the difference of two solutions  of equation (\ref{eqs}) satisfies the maximum principle.  As a consequence, we give a statement of  the comparison principle in our context. First, equation (\ref{eqs}) can be expressed as  $Q[u]=0$, where $Q$ is the operator     
\begin{equation}\label{eq4}
Q[u]=(1+|Du|^2)\Delta u-u_iu_ju_{i;j}-(1+|Du|^2),
\end{equation}
being $u_i=\partial u/\partial x_i$, $i=1,2$, and we assume the summation convention of repeated indices. The comparison principle asserts (\cite[Th. 10.1]{gt}):
 
 \begin{proposition}[Comparison principle] If $u,v\in C^2(\Omega)\cap C^0(\overline{\Omega})$ satisfy $Q[u]\geq Q[v]$ in $\Omega$ and $u\leq v$ on $\partial\Omega$, then $u\leq v$ in $\Omega$.  If we replace $Q[u]\geq Q[v]$ by $Q[u]> Q[v]$, then   $u<v$ in $\Omega$.
 \end{proposition}
 
 As a consequence, we deduce: 
 
  \begin{proposition}[Touching principle]\label{p-tan} Let $\Sigma_1$ and $\Sigma_2$ be two translating solitons with possibly non-empty boundaries $\partial\Sigma_1$, $\partial\Sigma_2$.   If $\Sigma_1$  and $\Sigma_2$ have  a common tangent interior point and $\Sigma_1$ lies above $\Sigma_2$ around $p$, then $\Sigma_1$ and $\Sigma_2$ coincide at an open set around $p$. The same statement is also valid if $p$ is a common boundary point and the tangent lines to $\partial\Sigma_i$ coincide at $p$.
\end{proposition}

 The tangency principle allows to control the shape of a given translating soliton by comparing, if possible, with other known surfaces   (\cite{lo1,lo3}). For instance,  it is easy to deduce        that   there do not exist closed translating solitons (Proposition \ref{pr-closed}). For this purpose, let  $\Sigma$ be a such surface. Take a vertical plane $\Pi$, which is a translating soliton, far from $\Sigma$ so $\Sigma\cap \Pi=\emptyset$ since $\Sigma$ is a compact set. Let us move $\Pi$ towards $\Sigma$ until the first touching point, which occurs necessarily at some interior point because $\partial\Sigma=\emptyset$. Then  the tangency principle implies that $\Sigma$ is included in $\Pi$, which is impossible.

In the following proposition, we use the tangency principle for compact translating solitons. By virtue of Proposition \ref{pr-closed}, the boundary of  a compact translating soliton is not an empty set. We will see that the boundary of the surface determines, in some sense, the shape of the whole surface that spans. For instance, we characterize the compact translating solitons with circular boundary.  

\begin{proposition}\label{pr-s1}
  Let $\Sigma$ be a compact translating soliton with boundary   $\partial\Sigma$.
\begin{enumerate}
\item If $\partial\Sigma$ is a graph on $\partial\Omega$,  $\Omega\subset\r^2$   a bounded domain, then   $\Sigma$ is a graph on $\Omega$.
\item Let $D\subset\r^2$ be the domain bounded by convex hull of the orthogonal projection of $\partial\Sigma$ on $\r^2$. Then  $\Sigma$ is contained in the solid cylinder $D\times\r$.
\item The maximum of the height of $\Sigma$ is attained at some boundary point, that is,  $\max_{p\in\Sigma}z(p)=\max_{p\in\partial\Sigma}z(p)$.
\end{enumerate}
As a consequence, if $\partial\Sigma$ is   a circle contained in a horizontal plane, then $\Sigma$ is a rotational surface  contained in a bowl soliton (\cite{pe,pyo}).
\end{proposition}

\begin{proof}
\begin{enumerate}
\item Suppose, contrary to our claim,   that $\Sigma$ is not a graph on $\Sigma$, in particular, there are two distinct points $p,q\in \mbox{int}(\Sigma)$ such that their orthogonal projections  coincide on $\r^2$. Let $\Sigma^t=\Sigma+t\vec{a}$ be a vertical translation of $\Sigma$ by the vector $t\vec{a}$. Move up $\Sigma$ sufficiently far so $\Sigma^t\cap\Sigma=\emptyset$ for $t$ sufficiently large. Now we come back $\Sigma^t$ by letting $t\searrow 0$ until the first time $t_1$ such that $\Sigma^{t_1}\cap\Sigma\not=\emptyset$. The existence of the points $p$ and $q$ ensures that $t_1>0$ and that this intersection occurs at some common interior point of both surfaces. By the tangency principle, $\Sigma^{t_1}=\Sigma$, a contradiction because their boundaries, namely,  $\partial\Sigma^{t_1}=\partial\Sigma+t_1\vec{a}$ and $\partial\Sigma$, do not coincide because $t_1\not=0$.
\item Let $v\in\r^3$ be a fixed arbitrary horizontal direction. Consider a vertical plane $\Pi$ and orthogonal to $v$. Take $\Pi$ sufficiently far so $\Sigma\cap\Pi=\emptyset$. We move $\Pi$ along the direction $v$ towards $\Sigma$ until the first touching point. By the tangency principle, the intersection must occur at some boundary point of $\Sigma$. By repeating this argument for all horizontal vectors, we conclude the proof.
\item Consider a horizontal plane $P$ above $\Sigma$ and sufficiently far so $\Sigma\cap P=\emptyset$. We move dow $P$ until the first touching point $p=(x_0,y_0,z_0)$ with $\Sigma$ at the height $t_1$. The proof is completed by showing that $p\in\partial\Sigma$. By contradiction, suppose that $p$ is an interior point of $\Sigma$. Consider $P$ as the graph of the function $v(x,y)=t_1$. Similarly, consider $\Sigma$ locally as the graph of a function $u$ around $p$ on some domain $\Omega\subset\r^2$. Then we have 
$Q[u]=0$, $Q[v]=-1$, so $Q[v]<Q[u]$. In view of $u\leq v$ on $\partial\Omega$ because $\Sigma$ lies below $P$,   the comparison principle implies $u<v$ in $\Omega$: a contradiction because $u(x_0,y_0)=v(x_0,y_0)$.  
\end{enumerate}
The proof of the last statement is as follows. By   items (1) and (3), $\Sigma$ is a graph on the round disc $\Omega$ bounded by $\partial\Sigma$ and the interior of $\Sigma$ lies below the plane $P$ containing $\partial\Sigma$. Then $\Sigma\cup\Omega$ bounds a $3$-domain. 
By using the technique of the Alexandrov reflection by vertical planes (\cite{al}), it is straightforward to see that  $\Sigma$ is invariant by any rotation whose axis is the vertical line through the center of $\Omega$. Accordingly, $\Sigma$ a surface of revolution, and since its boundary is a circle, then $\Sigma$ is contained in a bowl soliton.
\end{proof}

\begin{remark}\label{rem1} The last statement of the above proposition   gives other   argument for the existence of the bowl soliton. Indeed,  let $\Omega=D_r=\{(x,y)\in\r^2: x^2+y^2<r^2\}$ in (\ref{eqs}) and take the boundary data $\varphi=0$ in (\ref{eqsb}). By Corollary \ref{c1}, the existence and uniqueness of (\ref{eqs})-(\ref{eqsb}) is assured and Proposition  \ref{pr-s1} asserts that the solution is a radial function. Because the rotation axis meets orthogonally the domain $D_r$, then $\Sigma$ is a surface of revolution intersecting  orthogonally the $z$-axis.
\end{remark}

\begin{remark}[Tangency principle]\label{rem2} An inspection of  the comparison argument in the proof of   item (3) in Proposition \ref{pr-s1}  allows to extend the   tangency principle as follows.  Let $\Sigma_1$ and $\Sigma_2$ be two surfaces with weighted mean curvature $H_\phi^1$ and $H_\phi^2$, respectively.    Suppose that  $\Sigma_1$  and $\Sigma_2$ have  a common tangent interior point $p$ and the orientations in both surfaces coincide at $p$. If $H_\phi^1\leq H_\phi^2$ around $p$, and   $\Sigma_2$ lies above $\Sigma_1$ around $p$ with respect to $N(p)$, then $\Sigma_1$ and $\Sigma_2$ coincide at an open set around $p$. The same statement holds if $p$ is a common boundary point and the tangent lines to $\partial\Sigma_i$ coincide at $p$.
\end{remark}

We derive height   and interior gradient estimates for a solution of the translating soliton equation.

  \begin{proposition} \label{pr-31}
   Let $\Omega\subset\r^2$ be a bounded domain.  
  \begin{enumerate}
  \item The solution of (\ref{eqs})-(\ref{eqsb}), if exists, is unique.
  \item  There is a constant $C_1=C_1(\varphi, \Omega)$ such that if    $u$ is a solution of (\ref{eqs})-(\ref{eqsb}), then
\begin{equation}\label{eh}
C_1\leq u\leq \max_{\partial\Omega}\varphi\quad \mbox{in $\Omega$}.
\end{equation}
  \item  If $u$ is a solution of (\ref{eqs})-(\ref{eqsb}), then 
    $$\sup_{ \Omega}|Du|=\max_{\partial\Omega}|Du|.$$
\end{enumerate}
\end{proposition} 
   
   \begin{proof}
  \begin{enumerate}
  \item The uniqueness of solutions of (\ref{eqs})-(\ref{eqsb}) is a consequence of the maximun principle. 
    \item The inequality in the right-hand side of (\ref{eh}) is immediate from the   item (3) of Proposition \ref{pr-s1}.
  
The lower estimate for $u$ in (\ref{eh})  is obtained by means of   bowl soliton   as comparison surfaces. Let $R>0$ be sufficiently large so $\overline{\Omega}\subset D_R$. Let  $\mathcal{B}$ be a bowl soliton defined by a radial function ${\bf b}={\bf b}(r)$ such that $\partial D_R\subset\mathcal{B}$, that is,   ${\bf b}$ is a solution of (\ref{eqs}) in $D_R$ with ${\bf b}=0$ on $\partial D_R$. Let $\mathcal{B}_R$ denote the compact portion of $\mathcal{B}$ below the plane $z=0$.  Move vertically down $\mathcal{B}_R$ sufficiently far so $\Sigma_u$ lies above $\mathcal{B}_R$, that is, if $(x,y,z)\in\Sigma_u$, $(x,y,z')\in\mathcal{B}_R$, then $z>z'$. Then   move up $\mathcal{B}_R$ until the first touching point with $\Sigma_u$. If the first contact occurs at some interior point, then the touching principle implies   $\Sigma_u\subset \mathcal{B}_R$. The other possibility is that  the first contact point occurs when $\mathcal{B}_R$ touches a boundary point of $\Sigma_u$. In both cases, we  conclude $\mathbf{b}(0)\leq u-\min_{\partial\Omega}\varphi$ and consequently, the constant   $C_1= \mathbf{b}(0)+\min_{\partial\Omega}\varphi$ satisfies $C_1\leq u$.
  \item   Define the function  $v^i=u_i$, $i=1,2$, and   differentiate (\ref{eq4}) with respect to the variable $x_k$, obtaining 
 \begin{equation}\label{eq3}
 \left((1+|Du|^2)\delta_{ij}-u_iu_j\right)v_{i;j}^k+2\left(u_i\Delta u-u_ju_{i;j}-u_i \right)v_i^k=0,
 \end{equation}
for each $k=1,2$. Hence $v^k$ satisfies a linear elliptic equation and by the maximum principle,  $|v^k|$ has not a maximum at some interior point. Consequently,   the maximum of $|Du|$ on the compact set $\overline{\Omega}$ is attained at some boundary point.
 
  \end{enumerate}
  \end{proof}

\section{The Dirichlet problem in  bounded convex domains}\label{sec4}

In this section we prove Corollary \ref{c1}. Recall that the existence result of Serrin is also valid  for the general family of equations (\ref{eqs3}). By completeness of this paper,  we do a proof focusing on (\ref{eqs}) and following ideas of \cite{se}. We apply   the method of continuity which requires  the existence of    a priori $C^0$ and $C^1$ estimates for a solution  in order to   provide the necessary compactness properties.  These will be derived proving that  $u$  admits   barriers from above and from below along $\partial\Omega$. The higher order regularities of solutions hold under assuming smoothness hypothesis: \cite[Ths. 6.17, 6.19, 13.8]{gt}.

\begin{theorem}\label{t-ex}Let $\Omega\subset\r^2$ be a bounded   $C^{2,\alpha}$-domain whose inward satisfies   $\kappa\geq 0$. If $\varphi\in C^0(\partial\Omega)$, then there is a unique solution of  (\ref{eqs})-(\ref{eqsb}).
\end{theorem}

In  Proposition \ref{pr-31}, we found height estimates for $u$ and we proved that the interior gradient estimates are obtained once we have gradient estimates of $u$ along $\partial\Omega$. Thus, we now establish these estimates on the boundary.

     \begin{proposition}[Boundary gradient estimates]  \label{pr42}
   Let $\Omega\subset\r^2$ be a  bounded  domain with $C^2$-boundary,  $\kappa\geq 0$ and let $\varphi\in C^2(\partial\Omega)$. If  $u\in C^2(\Omega)\cap C^1(\overline{\Omega})$ is a   solution of (\ref{eqs})-(\ref{eqsb}), then there is a constant $C_2=C_2(\Omega, C_1,\|\varphi\|_{2;\overline{{\mathcal N}_\epsilon}})$ such that 
   $$\max_{\partial\Omega}|Du|\leq C_2,$$
    where $\varphi$ is extended to some tubular neighborhood ${\mathcal N}_\epsilon$ of $\partial\Omega$.
\end{proposition}

 \begin{proof} 
 We consider the operator $Q[u]$ defined in (\ref{eq4}), which we write now as
 \begin{equation}\label{eq44}
 Q[u]=a_{ij}u_{i;j}-(1+|Du|^2),\quad a_{ij}=(1+|Du|^2)\delta_{ij}-u_iu_j.
 \end{equation} An upper barrier for $u$  is obtained by considering the  solution $v^0$ of the Dirichlet problem for the minimal surface equation in $\Omega$ with the same boundary data $\varphi$: the existence of $v^0$ is assured in  (\cite{se}). Because $Q[v^0]<0=Q[u]$ and $v^0=u$ on $\partial\Omega$, we conclude  $v^0>u$ in $\Omega$ by the comparison principle.

  We now find a lower barrier for $u$. Here we use  the distance function in a small tubular neighborhood of $\partial\Omega$ in $\Omega$.  
  Consider on $\overline{\Omega}$ the distance function to $\partial\Omega$,  $d(x)=\mbox{dist}(x,\partial\Omega)$ and let $\epsilon>0$ be sufficiently small so  $\mathcal{N}_\epsilon=\{x\in\overline{\Omega}: d(x)<\epsilon\}$ is a tubular neighborhood of $\partial\Omega$.   We   parametrize $\mathcal{N}_\epsilon$  using normal coordinates $x\equiv (t,\pi(x)) \in\mathcal{N}_\epsilon$, where we write $x=\pi(x)+t\nu(\pi(x))$ for some $t\in [0,\epsilon)$,   $\pi:\mathcal{N}_\epsilon\rightarrow\partial\Omega$ is the orthogonal projection and $\nu$ is the unit   normal vector to $\partial\Omega$ pointing to $\Omega$. Among the properties of the function $d$, we know that $d$ is   $C^2$,  $|Dd|(x)=1$,  and $\Delta d \leq -\kappa(\pi(x))$  for all $x\in\mathcal{N}_\epsilon$.  

We extend $\varphi$ on $\mathcal{N}_\epsilon$ by   $\varphi(x)=\varphi(\pi(x))$. Define in $\mathcal{N}_\epsilon$ the function 
$$w=-h\circ d+\varphi,$$
 where  
  $$h(t)=a\log(1+bt),\quad a=\frac{c}{\log(1+b)},$$
where $b>0$  will be chosen later. Here    $c$ is any constant with
\begin{equation}\label{cc}
c>2\left(\|\varphi\|_{0}-C_1\right),
\end{equation}
and $C_1$ is the constant of   (\ref{eh}). Here and subsequently, $\|\cdot\|$ denotes the norm computed in $\overline{\mathcal{N}_\epsilon}$. It is immediate that $h\in C^\infty([0,\infty))$, $h'>0$ and $h''=-h'^2/a$. The   first and second derivatives of $w$ are $w_i=-h'd_i+\varphi_i$ and $w_{i;j}=-h''d_id_j-h'd_{i;j}+\varphi_{i;j}$. The computation of  $Q[w]$ leads to
\begin{equation}\label{q1}
Q[w]=-h''a_{ij}d_id_j-h'a_{ij}d_{i;j}+a_{ij}\varphi_{i;j}-(1+|Dw|^2). 
\end{equation}
  From $|Dd|=1$, it follows that $\langle D(Dd)_x\xi,Dd(x)\rangle=0$ for all $\xi\in\r^2$. If  $\{v_1,v_2\}$ is the canonical basis of $\r^2$, by taking $\xi=v_i$, we find $d_{i;j}d_j=0$. Thus
\begin{eqnarray*}
w_iw_jd_{i;j}&=&(-h'd_i+\varphi_i)(-h'd_j+\varphi_j)d_{i;j}=(h'^2d_i-2h'\varphi_i)d_jd_{i;j}+\varphi_i \varphi_jd_{i;j}\\
&=&\varphi_{i}\varphi_jd_{i;j}\geq |D\varphi|^2\Delta d,
\end{eqnarray*}
where the last inequality is due to $D^2d$ is negative. Using this inequality and   the definition of $a_{ij}$ in (\ref{eq44}), we derive
\begin{equation}\label{qa}
a_{ij}d_{i;j}=(1+|Dw|^2)\Delta d-w_iw_j d_{i;j}\leq(1+|Dw|^2-|D\varphi|^2)\Delta d.
\end{equation}
 Notice that 
 \begin{equation}\label{qc}
 |Dw|^2=h'^2+|D\varphi|^2-2h' \langle Dd, D\varphi\rangle.
 \end{equation}
 Then
\begin{eqnarray*}
1+|Dw|^2-|D\varphi|^2&=&1+h'^2-2h' \langle D\varphi, Dd\rangle\geq 1+h'^2-2h'|D\varphi|\\
&=&1+\frac{c^2b^2}{\log(1+b)^2(1+bt)^2}-\frac{2cb}{\log(1+b)(1+bt)}|D\varphi|>0
\end{eqnarray*}
 if $b$ is sufficiently large, with $b$ a constant depending on $\partial\Omega$, $c$ and $|D\varphi|$. Since $\Delta d\leq 0$ because $D^2 d$ is negative, we deduce from (\ref{qa}) that  $a_{ij}d_{i;j}\leq 0$. 
 
 The ellipticity of $A=(a_{ij})$ can be written as   $|\xi|^2\leq a_{ij}\xi_i\xi_j\leq(1+|Dw|^2)|\xi|^2$ for all $\xi\in\r^2$. Taking $\xi=Dd$, then $1=|Dd|^2\leq a_{ij}d_id_j\leq (1+|Dw|^2)$.   Since $h''<0$, we have  
 \begin{equation}\label{qb}
 h''(a_{ij}d_id_j)\leq h''.
 \end{equation}
  On the other hand, if $\cdot$ is the usual scalar product in the set of the square matrix, 
   $$|A|^2=  A\cdot A=1+(1+|Dw|^2)^2\leq 2(1+|Dw|^2)^2,$$ 
   hence
 $$a_{ij}\varphi_{i;j}=  A\cdot D^2\varphi \geq -|A| |D^2\varphi|\geq -\sqrt{2} |D^2\varphi|(1+|Dw|^2).$$ 
 By combining this inequality with $\Delta d\leq 0$, and inserting (\ref{qa}) and (\ref{qb}) in (\ref{q1}),    we deduce 
\begin{eqnarray*}
Q[w]&\geq& -h''-h'(1+|Dw|^2-|D\varphi|^2)\Delta d-(1+\sqrt{2}|D^2 \varphi|)(1+|Dw|^2)\\
&\geq &-h''-(1+\sqrt{2}|D^2 \varphi|)(1+|Dw|^2)\geq -h''-\beta(1+|Dw|^2),
\end{eqnarray*}
where $\beta=1+\sqrt{2}\| D^2\varphi\|_{0}$. Take $b$ sufficiently large if necessary, to ensure that $1/a-\beta>0$, so $\beta$ depends on $\|D^2\varphi\|_{0}$, $C_1$ and $\|\varphi\|_0$. Using that   $h''=-h'^2/a$ and (\ref{qc}), we obtain 
\begin{eqnarray}\label{q4}
Q[w]&\geq& \frac{h'^2}{a}-\beta(1+|Dw|^2)=\left(\frac{1}{a}-\beta\right)h'^2+2h'\beta \langle Dd, D\varphi\rangle   -\beta(1+\| D\varphi\|_{0}^2)\nonumber\\
&\geq &\left(\frac{1}{a}-\beta\right)h'^2-2h'\beta\| D\varphi\|_{0}-\beta(1+\| D\varphi\|_{0}^2)\nonumber\\
&=&\left(\frac{1}{a}-\beta\right)\frac{a^2b^2}{(1+bt)^2}-2\beta\frac{ab}{1+bt}  \|D\varphi\|_{0} -\beta(1+\|D\varphi\|_{0}^2).
\end{eqnarray}
We write the  last term  as a function on $\mathcal{N}_\epsilon$, namely, $g(x)=g(t,\pi(x))$. At $t=0$, 
\begin{eqnarray*}
g(0)&=&\left(\frac{1}{a}-\beta\right)\frac{c^2b^2}{\log(1+b)^2}-2\beta\frac{cb}{\log(1+b)}  \|D\varphi\|_{0} -\beta(1+\|D\varphi\|_{0}^2)\\
&=&\frac{cb}{\log(1+b)}\left( \left(\frac{1}{a}-\beta\right)\frac{cb}{\log(1+b)}-2\beta \|D\varphi\|_{0} \right)-\beta(1+\|D\varphi\|_{0}^2)
\end{eqnarray*}
Therefore, if $b$ is sufficiently large, $g(0)>0$.  Since $\partial\Omega$ is compact, by an argument of continuity, $b$ can be  chosen sufficiently large to ensure that  $g(t)>0$ in $\mathcal{N}_\epsilon$.   For this choice of $b$, we find $Q[w]>0$.

 In order to assure that $w$ is a local lower barrier  in $\mathcal{N}_\epsilon$, we have to see  that 
\begin{equation}\label{mm}
w\leq u\quad \mbox{in $\partial\mathcal{N}_\epsilon$}.
\end{equation}
   In $\partial\mathcal{N}_\epsilon\cap\partial\Omega$,   the distance function  is $d=0$, so $w=\varphi=u$. On the other hand, let us further require $b$ large enough  so  $\log(1+b\epsilon)/\log(1+b)\geq 1/2$.  Then 
in $\partial\mathcal{N}_\epsilon\setminus\partial\Omega$, we find from (\ref{cc}) that 
\begin{eqnarray*}
w&=&-c\frac{\log(1+b\epsilon)}{\log(1+b)}+\varphi\leq -\frac{\|\varphi\|_{0}-C_1}{2}\frac{\log(1+b)}{\log(1+b\epsilon)}+\varphi\\
&\leq&   C_1-\|\varphi\|_{0}+\varphi\leq C_1\leq u
\end{eqnarray*}
   in  $\partial\mathcal{N}_\epsilon\setminus\partial\Omega$.   Definitively, (\ref{mm}) holds  in $\partial\mathcal{N}_\epsilon\setminus\partial\Omega$. Because $Q[w]>0=Q[u]$, we conclude     $w\leq w$ in $\mathcal{N}_\epsilon$ by the comparison principle. 
  
Consequently, we have proved the existence of lower and upper barriers for $u$  in $\mathcal{N}_\epsilon$, namely, $w\leq u\leq v^0$. Hence  
$$\max_{\partial\Omega}|Du|\leq C_2:=\max\{\|Dw\|_{0;\partial\Omega}, \|Dv^0\|_{0;\partial\Omega}\}$$
  and both values $\|Dw\|_{0;\partial\Omega}, \|Dv^0\|_{0;\partial\Omega}$ depend only on  $\Omega$, $C_1$ and $\varphi$.  This completes the proof of proposition. 
    \end{proof}
     
\begin{remark}
\begin{enumerate}
\item It is possible,   instead  the function $v^0$, to use   $w=h\circ d+\varphi$ for an  upper barrier of $u$. 
 \item The use of the auxiliary function $h(d)=a\log(1+bd)$ for obtaining boundary gradient estimates is standard in the theory of elliptic equations (see \cite[Ch. 14]{gt} as a general reference). It should also be
mentioned that Bernstein was the first author whose employed  this function    to construct barriers for solutions in elliptic equations in two variables, assuming analytic hypothesis: \cite[pp. 265-6 ]{be0}. 
 \end{enumerate}
 \end{remark}
 
\begin{proof}[of Theorem  \ref{t-ex}]
 
 In a first step,  we demonstrate the theorem when  $\varphi\in C^{2,\alpha}(\overline{\Omega})$. We establish the solvability  of the Dirichlet problem   (\ref{eqs})-(\ref{eqsb}) by applying a slightly modification of the   method of continuity  (\cite[Sec. 17.2]{gt}). Define  the family of Dirichlet  problems parametrized by $t\in [0,1]$ by  
 $$ \left\{\begin{array}{cll}
Q_t[u]&=&0 \mbox{ in $\Omega$}\\
 u&=& \varphi \mbox{ on $\partial\Omega,$}
 \end{array}\right.$$
 where 
   $$Q_t[u]= (1+|Du|^2)\Delta u-u_iu_ju_{i;j}- t(1+|Du|^2).$$
   As usual, let 
$$\mathcal{A}=\{t\in [0,1]: \exists u_t\in C^{2,\alpha}(\overline{\Omega}),   Q_t[u_t]=0, {u_t}_{|\partial\Omega}=\varphi\}.$$ 
The theorem is established if   $1\in \mathcal{A}$. For this purpose, we prove that $\mathcal{A}$ is a non-empty open and closed subset of $[0,1]$.

\begin{enumerate}
\item  The set  $\mathcal{A}$ is not empty. Let us observe that $0\in\mathcal{A}$ because  the minimal solution  $v^0$  defined in Proposition \ref{pr42} corresponds with $t=0$. 
\item The set $\mathcal{A}$ is open in $[0,1]$. Given $t_0\in\mathcal{A}$, we need to prove that there exists $\epsilon>0$ such that $(t_0-\epsilon,t_0+\epsilon)\cap [0,1]\subset\mathcal{A}$. Define the map $T(t,u)=Q_t[u]$ for $t\in\r$ and $u\in  C^{2,\alpha}(\overline{\Omega})$. Then $t_0\in\mathcal{A}$ if and only if $T(t_0,u_{t_0})=0$. If we show that the derivative  of $Q_t$ with respect to $u$, say $(DQ_t)_u$, at the point $u_{t_0}$ is an isomorphism, it follows from the Implicit Function Theorem the existence of an open set $\mathcal{V}\subset C^{2,\alpha}(\overline{\Omega})$, with $u_{t_0}\in \mathcal{V}$ and a $C^1$ function $\psi:(t_0-\epsilon,t_0+\epsilon)\rightarrow \mathcal{V}$ for some $\epsilon>0$, such that $\psi(t_0)=u_{t_0}>0$ and  $T(t,\psi(t))=0$ for all $t\in (t_0-\epsilon,t_0+\epsilon)$: this guarantees that $\mathcal{A}$ is an open  set of  $[0,1]$.

To show that $(DQ_t)_u$ is one-to-one is equivalent that say that for any $f\in C^\alpha(\overline{\Omega})$, there is a unique solution $v\in C^{2,\alpha}(\overline{\Omega})$ of the linear equation $Lv:=(DQ_t)_u(v)=f$ in $\Omega$ and $v=\varphi$ on $\partial\Omega$. The computation of $L$ was done in Proposition \ref{pr-31}, namely, 
$$Lv=(DQ_t)_uv=a_{ij}(Du)v_{i;j}+\mathcal{B}_i(Du,D^2u)v_i,$$
where $a_{ij}$ is as in (\ref{eq4}) and $\mathcal{B}_i=2(u_i \Delta u-u_ju_{i;j}-tu_i)$.
The existence and uniqueness is assured by standard theory (\cite[Th. 6.14]{gt}).

\item The set $\mathcal{A}$ is closed in $[0,1]$. Let $\{t_k\}\subset\mathcal{A}$ with $t_k\rightarrow t\in [0,1]$. For each $k\in\mathbb{N}$, there is $u_k\in C^{2,\alpha}(\overline{\Omega})$  such that $Q_{t_k}[u_k]=0$ in $\Omega$ and $u_k=\varphi$ in $\partial\Omega$. Define the set
$$\mathcal{S}=\{u\in C^{2,\alpha}(\overline{\Omega}): \exists t\in [0,1]\mbox{ such that }Q_{t}[u]=0 \mbox{ in }\Omega, u_{|\partial\Omega}=\varphi\}.$$
Then $\{u_k\}\subset\mathcal{S}$. If we see that the set $\mathcal{S}$ is bounded in $C^{1,\beta}(\overline{\Omega})$ for some $\beta\in[0,\alpha]$, and since $a_{ij}=a_{ij}(Du)$ in (\ref{eq4}), the Schauder theory proves that $\mathcal{S}$ is bounded in $C^{2,\beta}(\overline{\Omega})$, in particular, $\mathcal{S}$ is precompact in $C^2(\overline{\Omega})$  (Th. 6.6 and Lem. 6.36 in \cite{gt}). Hence there is a subsequence $\{u_{k_l}\}\subset\{u_k\}$ converging to some $u\in C^2(\overline{\Omega})$ in $C^2(\overline{\Omega})$. Since $T:[0,1]\times C^2(\overline{\Omega})\rightarrow C^0(\overline{\Omega})$ is continuous, we obtain  $Q_t[u]=T(t,u)=\lim_{l\rightarrow\infty}T(t_{k_l},u_{k_l})=0$ in $\Omega$. Moreover, $u_{|\partial\Omega}=\lim_{l\rightarrow\infty} {u_{k_l}}_{|\partial\Omega}=\varphi$ on $\partial\Omega$, so $u\in C^{2,\alpha}(\overline{\Omega})$ and consequently, $t\in \mathcal{A}$.

Definitively,   $\mathcal{A}$ is  closed in $[0,1]$ provided we find a constant $M$ independent on $t\in\mathcal{A}$,  such that  
$$
\|u_t\|_{C^1(\overline{\Omega})}=\sup_\Omega |u_t|+\sup_\Omega|Du_t|\leq M.
$$
Let $t_1<t_2$, $t_i\in [0,1]$, $i=1,2$. Then $Q_{t_1}[u_{t_1}]=0$ and 
$$Q_{t_1}[u_{t_2}]=(t_2-t_1) (1+|Du_{t_2}|^2)>0=Q_{t_1}[u_{t_1}].$$
  Since $u_{t_1}=u_{t_2}$ on $\partial\Omega$, the comparison principle yields $u_{t_2}<u_{t_1}$ in $\Omega$. This proves that the solutions $u_{t_i}$ are ordered in decreasing sense according the parameter $t$. It turns out that  $u_1\leq u_t<v^0$ for all $t$, where $u_1$ is the solution of (\ref{eqs})-(\ref{eqsb}). According to (\ref{eh2}), we have $C_1\leq  u_t \leq \sup_\Omega u_0\leq \max_{\partial\Omega}\varphi$ and we conclude  
\begin{equation}\label{ut}
 \| u_t\|_{0;\overline{\Omega}}\leq C_3,\quad C_3=\max\{|C_1|,\|\varphi\|_{0;\partial\Omega}\}.
\end{equation}

 In order to find  the desired gradient estimates for the solution $u_t$, by Proposition \ref{pr-31}, we have to find estimates of $|Du_t|$ along $\partial\Omega$. On the other hand, the same computations  given in   Proposition \ref{pr42}     conclude that $\sup_{\partial\Omega}|Du_t|$ is bounded by a constant depending on   $\Omega$, $\varphi$ and  $\|u_t\|_{0;\overline{\Omega}}$. However, and by using (\ref{ut}),  the value  $\|u_t\|_{0;\overline{\Omega}}$ is bounded by $C_3$, which depends only on   $\varphi$ and $\Omega$.   

\end{enumerate}
Until here, we have proved   the part of existence in Theorem \ref{t1}. The uniqueness is a consequence  of Proposition \ref{pr-31} and this completes the proof of theorem if $\varphi\in C^2(\partial\Omega)$.
 
 Finally we suppose $\varphi\in C^0(\partial\Omega)$.  Let $\{\varphi_k^{+}\}, \{\varphi_k^{-}\}\in C^{2,\alpha}(\partial\Omega)$ be a monotonic sequence of functions converging from above and from below to $\varphi$ in the $C^0$ norm. By virtue of the first part of this proof,   there are solutions $u_k^{+}, u_k^{-}\in C^{2,\alpha}(\overline{\Omega})$ of the translating soliton equation (\ref{eqs}) such that ${u_k^{+}}_{|\partial\Omega}=\varphi_k^{+}$ and ${u_k^{-}}_{|\partial\Omega}=\varphi_k^{-}$.  By the comparison principle, we find
$$u_1^{-}\leq\ldots \leq u_k^{-}\leq u_{k+1}^{-}\leq\ldots \leq u_{k+1}^{+}\leq u_k^+\leq\ldots\leq u_1^{+}$$
for every $k$, hence the sequences $\{u_k^{\pm}\}$ are uniformly bounded in the $C^0$ norm. By the proof of Theorem \ref{t-ex}, the sequences $\{u_k^{\pm}\}$ have a priori $C^1$ estimates depending only on  $\Omega$ and $\varphi$.    Using classical Schauder theory again (\cite[Th. 6.6]{gt}),  the sequence $\{u_k^{\pm}\}$ contains a subsequence  $\{v_k\}\in C^{2,\alpha}(\overline{\Omega})$ converging uniformly on the $C^2$ norm on compacts subsets of $\Omega$ to a solution $u\in C^2(\Omega)$ of (\ref{eqs}). Since $\{{u_k^{\pm}}_{|\partial\Omega}\}=\{\varphi_k^{\pm}\}$    and $\{\varphi_k^{\pm}\}$  converge to $\varphi$, we conclude  that $u$ extends continuously to $\overline{\Omega}$ and $u_{|\partial\Omega}=\varphi$.

\end{proof}

\section{The Dirichlet problem for the constant weighted mean curvature}\label{sec5}
In this section we solve the Dirichlet problem for the case that $H_\phi$ is constant in (\ref{hf}):
\begin{eqnarray}
&&\mbox{div} \frac{Du}{\sqrt{1+|Du|^2}} = \frac{1}{\sqrt{1+|Du|^2}}+\mu\quad \mbox{in $\Omega$}\label{eqm1}\\
&&u=\varphi\quad \mbox{on $\partial\Omega,$}\label{eqm2}
\end{eqnarray}
where $\mu$ is a constant and $u\in C^2(\Omega)\cap C^0(\overline{\Omega})$. The motivation of this problem is twofold. First, equation (\ref{eqm1}) is the analogous to the constant mean curvature equation in Euclidean space in the context of manifolds with density, whereas (\ref{eqs}) corresponds with the minimal surface equation. Second, the solvability of the constant mean curvature equation holds   for any $\varphi$ if $\kappa\geq 2|H|\geq 0$ (\cite{se}) and the next result for (\ref{eqm1})-(\ref{eqm2}) establishes a similar result.

\begin{theorem}\label{t-exm}Let $\Omega\subset\r^2$ be a  bounded   $C^{2,\alpha}$-domain with inward curvature $\kappa$. If $\kappa\geq \mu\geq 0$ and $\varphi\in C^0(\partial\Omega)$, then there is a unique solution of  (\ref{eqm1})-(\ref{eqm2}).
\end{theorem}

Notice that the solvability of this Dirichlet problem was proved in \cite{jl} in a context more general where the domain may be not bounded. As a difference, the present proof    uses a comparison argument with rotational surfaces in order to obtain the   $C^0$ estimates. The uniqueness is again a consequence of the maximum principle for the equation (\ref{eqm1}). After  Theorem \ref{t-ex}, we   assume that $\mu>0$. The proof now differs only in minor details from than of the preceding section, which  are left to the reader.  We point out that   the assumption $\mu>0$ will be use strongly.

  \begin{lemma} \label{le1}
 If  $\Omega\subset\r^2$ is  a bounded domain with $\mbox{diam}(\Omega)<1/\mu$, then there is a constant $C_1=C_1(\varphi, \Omega)$ such that if    $u$ is a solution of (\ref{eqm1})-(\ref{eqm2}), then
\begin{equation}\label{eh2}
C_1\leq u\leq \max_{\partial\Omega}\varphi\quad \mbox{in $\Omega$}.
\end{equation}
 \end{lemma} 
   
   \begin{proof} Using $\mu>0$,   the right-hand side of (\ref{eqm1}) is non-negative,  the maximum principle implies $\sup_\Omega u=\max_{\partial\Omega}u=\max_{\partial\Omega}\varphi$, proving the inequality in the right-hand side of (\ref{eh2}). The lower estimate for $u$ in (\ref{eh2})  is obtained by using  radial solutions of (\ref{eqm1}). It was proved in \cite{lo1} that if $\mu>0$, any radial solution   intersecting the rotational axis (and necessarily perpendicularly) converges to a right circular cylinder of radius $1/\mu$. More exactly, let $D_r\subset\r^2$ be a disc centered at the origin of radius $r$. If $\mu>1/2$, there is a radial solution of (\ref{eqm1}) on  $D_r$ for some   $r_0>1/\mu$ and if $0<\mu\leq 1/2$, there is a radial solution of (\ref{eqm1}) on  $D_r$ for any $r<1/\mu$.  
  
Since  $\mbox{diam}(\Omega)<1/\mu$, let $r>0$ such that $\mbox{diam}(\Omega)<r<1/\mu$ and denote by $v$ the radial solution of (\ref{eqm1})  on $D_r$ with $v=0$ on $\partial D_r$. After a horizontal translation if necessary, we suppose $\Omega\subset  D_r$. Now the argument works the same as in Proposition \ref{pr-31}  with the graph $\Sigma_v$, where now $C_1=v(0)+\min_{\partial\Omega}\varphi$.
    \end{proof}

 \begin{lemma}[Interior gradient estimates] \label{le2} If $u$ is a solution of (\ref{eqm1})-(\ref{eqm2}), then 
  $$\sup_{ \Omega}|Du|=\max_{\partial\Omega}|Du|.$$
 \end{lemma}
 \begin{proof}
Now the corresponding equation (\ref{eq3}) for (\ref{eqm1}) is
\begin{equation}\label{eqv}
\left( \left( 1+| \nabla v| ^{2}\right) \delta
_{ij}-v_{i}v_{j}\right) z_{i;j}^{k}+2\left( v_{i}\Delta v-v_{i;j}u_{i}-v_i-\frac32\mu(1+|\nabla v|^2)\right) z_{i}^{k}=0,
\end{equation}
and the arguments are similar.
\end{proof}


      \begin{lemma}[Boundary gradient estimates]  \label{le3}
   Let $\Omega\subset\r^2$ be a  bounded  domain with $C^2$-boundary, $\kappa\geq \mu> 0$ and let $\varphi\in C^2(\partial\Omega)$. If  $u\in C^2(\Omega)\cap C^1(\overline{\Omega})$ is a   solution of (\ref{eqm1})-(\ref{eqm2}), then there is a constant $C_2=C_2(\Omega, C_1,\|\varphi\|_{2})$ such that 
   $$\max_{\partial\Omega}|Du|\leq C_2.$$
    
\end{lemma}

 \begin{proof} 
 We consider the operator 
  \begin{equation}\label{eq444}
 Q[u]=a_{ij}u_{i;j}-(1+|Du|^2)-\mu(1+|Du|^2)^{3/2}.
 \end{equation} 
 The minimal solution $v^0$ is   an upper barrier for $u$. For the lower barrier for $u$, we use again   the   function $w=-h\circ d+\varphi$. Now 
\begin{equation}\label{qq1}
Q[w]=-h''a_{ij}d_id_j-h'a_{ij}d_{i;j}+a_{ij}\varphi_{i;j}-(1+|Dw|^2)-\mu(1+|Dw|^2)^{3/2}.
\end{equation}
Taking into account $(1+|Dw|^2)^{1/2}\leq 1+|Dw|$ and $|Dw|^2\leq h'^2+|D\varphi|^2+2h'|D\varphi|\leq (h'+|D\varphi|)^2$, we deduce from (\ref{qq1})
\begin{eqnarray*}
Q[w]&\geq& -h''-h'(1+|Dw|^2-|D\varphi|^2)\Delta d-(1+\sqrt{2}|D^2 \varphi|)(1+|Dw|^2)-\mu(1+|Dw|^2)^{3/2}\\
&\geq &-h''-h'(1+|Dw|^2-|D\varphi|^2)\Delta d-(1+\sqrt{2}|D^2 \varphi|)(1+|Dw|^2)-\mu(1+|Dw|^2)(1+|Dw|) \\
&\geq &-h''-h'(1+|Dw|^2-|D\varphi|^2)(\Delta d+\mu)-\mu h'|D\varphi|^2\\
&-&(\mu(1+|D\varphi|)+1+\sqrt{2}|D^2 \varphi|)(1+|Dw|^2).
\end{eqnarray*}
Let $\beta=\mu(1+\|D\varphi\|_{0})+1+\sqrt{2}\| D^2\varphi\|_{0}$. Since $\Delta d+\mu\leq -\kappa+\mu\leq 0$, it follows  
\begin{eqnarray*} 
Q[w]&\geq& \frac{h'^2}{a}-\beta(1+|Dw|^2)-\mu h'|D\varphi|^2\\
&\geq&\left(\frac{1}{a}-\beta\right)h'^2-h'(2\beta\| D\varphi\|_{0}+\mu\|D\varphi\|^2_{0})-\beta(1+\| D\varphi\|_{0}^2).\end{eqnarray*}  

The rest of the proof runs   as in Proposition \ref{pr42}.
    \end{proof}
     
  With the help of the preceding three lemmas we can now prove Theorem \ref{t-exm}.
\begin{proof}[of Theorem  \ref{t-exm}]
 
For the method of continuity, let
   $$Q_t[u]= (1+|Du|^2)\Delta u-u_iu_ju_{i;j}- (1+|Du|^2)-t\mu (1+|Du|^2)^{3/2},$$
and 
$$\mathcal{A}=\{t\in [0,1]: \exists u_t\in C^{2,\alpha}(\overline{\Omega}),  Q_t[u_t]=0, {u_t}_{|\partial\Omega}=\varphi\}.$$ 
 The set  $\mathcal{A}$ is not empty  because    the   solution of Theorem \ref{t-ex} corresponds with the value $t=0$. For the openness of   $\mathcal{A}$,    the  computation of $L$ leads to 
$$Lv=(DQ_t)_uv=a_{ij}(Du)v_{i;j}+\mathcal{B}_i(Du,D^2u)v_i,$$
where  $\mathcal{B}_i=2(u_i \Delta u-u_ju_{i;j}-u_i-3t(1+|Du|^2)/2)$. Then the proof works again. 

Finally, we show  that the set $\mathcal{A}$ is closed in $[0,1]$. For the height and gradient estimates for $u_t$, we use lemmas \ref{le1}, \ref{le2} and \ref{le3}. The arguments are similar once we prove  that the solutions $u_{t_i}$ are ordered in decreasing sense. If $t_1<t_2$, then $Q_{t_1}[u_{t_1}]=0$ and 
$$Q_{t_1}[u_{t_2}]=(t_2-t_1) \mu(1+|Du_{t_2}|^2)^{3/2}>0=Q_{t_1}[u_{t_1}].$$
 Since $u_{t_1}=u_{t_2}$ on $\partial\Omega$, the comparison principle yields $u_{t_2}<u_{t_1}$ in $\Omega$.  
\end{proof}

\section{The Dirichlet problem in unbounded domains}\label{sec6}

We study in this section the Dirichlet problem (\ref{eqs})-(\ref{eqsb}) in unbounded convex  domains contained   in a strip. We have two cases depending if the domain is or is not a strip. 

The first result assumes that $\Omega$ is a strip. In such a case, the motivation comes from   the grim reapers that appeared in (\ref{ws}). For each $\theta$, the surface $\Sigma_{w_\theta}$ is a graph defined in the (maximal) strip $\Omega^\theta$, with $w_\theta(x,y)\rightarrow +\infty$ as $|y|\rightarrow \pi/(2\cos\theta)$. If we narrow the strip to   $|y|<b$, with $0<b<\pi/(2\cos\theta)$, then the value of $w_\theta$ on $|y|=b$ is the linear function  $x\mapsto \varphi(x,\pm b)=w_\theta(x,b)$ and 
$\partial \Sigma_{w_\theta}$   is formed  by two parallel straight lines.

Our purpose is to consider the Dirichlet problem  when $\Omega$ is a strip and   $\varphi$ is formed by two copies of a convex  function.  Let   $\Omega_m=\{(x,y)\in\r^2:-m<y<m\}$, $m>0$, be the strip of width $2m$. For each  smooth convex function  $f$  defined in $\r$,    we extend $f$ to a function $\varphi_f$ on $\partial\Omega_m$  by  $\varphi_f(x,\pm m)=f(x)$. The result of existence is established by our next theorem (\cite{lo4}).

\begin{theorem}\label{t-strip}   If $m<\pi/2$, then  for each convex function $f$, there is a   solution of (\ref{eqs})-(\ref{eqsb}) for boundary values $\varphi_f$ on $\partial\Omega_m$.
\end{theorem}

The proof uses the classical Perron
method of sub and supersolutions: see \cite[pp. 306-312]{ch}, \cite[Sec. 6.3]{gt}).
We consider the operator $Q$ defined in  (\ref{eq4}), where we know that $Q[u]=0$ if and only if $u$ is a solution of the translating soliton equation. The existence result of Theorem \ref{t-ex} holds  in disks, so we can proceed to apply the Perron process when the domain is a strip. 

First  we need   a subsolution of (\ref{eqs})-(\ref{eqsb}). In the following result, $f$ is not necessarily a convex function (\cite{co}).

\begin{proposition}\label{p-min} Let $\Omega_m\subset\r^2$ be a strip. If $f$ is a continuous function defined in $\r$, then there is a solution $v^0$ of the Dirichlet problem
\begin{equation}\label{em1}
\begin{split} &\mbox{\rm div}\left(\dfrac{Du}{\sqrt{1+|Du|^2}}\right)=0 \quad \mbox{in $\Omega_m$}\\
&u=\varphi_f\quad \mbox{on $\partial\Omega_m$}
\end{split}
\end{equation}
with the property $f(x)<v^0(x,y)$ for all $(x,y)\in\Omega_m$.
\end{proposition}

 Let $u\in C^0(\overline{\Omega_m})$ be a continuous function and let $D$ be a closed round disk in $\Omega_m$. We denote by $\bar{u}\in C^2(D)$ the unique solution of the Dirichlet problem
 $$
 \left\{\begin{array}{ll}
 Q[\bar{u}]=0 & \mbox{ in $D$}\\
 \bar{u}=u & \mbox{ on $\partial D,$}
 \end{array}\right.$$
whose existence and uniqueness is assured by  Theorem \ref{t-ex}. We extend $\bar{u}$   to $\Omega_m$ by continuity as 
 $$M_D[u]=\left\{\begin{array}{ll} 
 \bar{u}& \mbox{ in } D\\
 u&\mbox{ in }\Omega_m\setminus D.
 \end{array}\right.$$
 The function $u$ is said to be a {\it supersolution} in $\Omega_m$ if    $M_D[u]\leq u$ for every closed round disk $D$ in $\Omega_m$.  For example, for any domain $\Omega\subset\r^2$, the function $u=0$ in $\overline{\Omega}$ is a supersolution in $\Omega$. Indeed, if $D\subset \Omega$ is a closed round disk, then $\bar{u}<0$ since $Q[0]=-1<0=Q[\bar{u}]$ and the comparison principle applies. Thus $M_D[u]\leq 0$. 
 
On the other hand,   for each $p\in\Omega$, there is a supersolution $u$ with $u(p)<0$. To this end, consider $D\subset\Omega$   a closed round disk centered at the origin of $\r^2$. Let $\mathbf{b}=\mathbf{b}(r)$ be the bowl soliton with $\mathbf{b}_{|\partial D}=0$. Then the function   $u$ defined as $u=\mathbf{b}$ in $D$ and $u=0$ in $\overline{\Omega}\setminus D$ is a supersolution.
  
 \begin{definition} A function $u\in C^0(\overline{\Omega_m})$ is called a   superfunction relative to $f$ if $u$  is a supersolution in $\Omega_m$ and $f\leq u$ on $\partial \Omega_m$. Denote
 by $\mathcal{S}_f$   the class of all superfunctions relative to $f$,  
 $$\mathcal{S}_f=\{u\in C^0(\overline{\Omega_m}):M_D[u]\leq u \mbox{ for every closed  disk $D\subset\Omega_m$}, f\leq u \mbox{ on $\partial\Omega_m$}\}.$$ 
 \end{definition}

 \begin{lemma}   The set $\mathcal{S}_f$  is not empty.
 \end{lemma}
 
 \begin{proof} We claim that $v^0\in  \mathcal{S}_f$, where  $v^0$ is  the minimal solution  given in Proposition \ref{p-min}. Let $D\subset\Omega_m$ be a closed round disk. Since $v^0$ is a minimal surface,   $Q[v^0]=-(1+|Dv^0|^2)<0$ and because $\overline{v^0}=v^0$ in $\partial D$, the comparison principle implies $M_D[v^0]=\overline{v^0}\leq v^0$ in $D$. On the other hand, $v^0=f$ on $\partial\Omega_m$, proving definitively that $v^0\in \mathcal{S}_f$. 
 \end{proof}

We now give some  properties about superfunctions  whose proofs are straightforward: in the case of the constant mean curvature equation, we refer \cite{lo}; in the context of  translating solitons, see \cite[Lems. 4.2--4.4]{jj}.  

 \begin{lemma}\label{p-p}
  \begin{enumerate}
 \item If $\{u_1,\ldots,u_n\}\subset  \mathcal{S}_f$, then $\min\{u_1,\ldots,u_n\}\in \mathcal{S}_f$.
 \item The operator $M_D$ is increasing in $\mathcal{S}_f$.
 \item If $u\in \mathcal{S}_f$ and  $D$ is a closed round disk in $\Omega_m$, then $M_D[u]\in\mathcal{S}_f$.
 \end{enumerate}
 \end{lemma}

 Consider the family of  grim reaper $w_\theta$ of (\ref{ws}). Since $w_\theta$ is defined   in the strip  $\Omega^{\theta}$ and, by assumption,  $m<\pi/2$, then $\Omega_m\subset\Omega^0\subset \Omega^{\theta}$ for any $\theta$. Thus it makes sense to restrict $w_\theta$ to the strip $\Omega_m$ and  we keep the same notation for its restriction in $\Omega_m$.  Consequently $w_\theta$ is a linear function on $\partial\Omega_m$ and $\partial \Sigma_{w_\theta}$  consists of two parallel lines.

Consider the subfamily of      grim reapers 
$$\mathcal{G}=\{w_\theta: w_\theta\leq f\mbox{ on $\partial\Omega_m$},\theta\in(-\pi/2,\pi/2)\}.$$ 
Notice that the set $\mathcal{G}$ is not empty because $f$ is convex. Furthermore, the minimal surface   $v^0$  with $v^0=f$ on $\partial\Omega_m$ satisfies   $Q[v^0]<0=Q[w_\theta]=0$ for all $w_\theta\in\mathcal{G}$. Hence,  the comparison principle asserts that $w_\theta<v^0$ in $\Omega_m$ for all $\theta\in (-\pi/2,\pi/2)$. This implies that $v^0$   plays the role of a subsolution for (\ref{eqs})-(\ref{eqsb}).

We  now construct a solution of equation (\ref{eqs})  between the  grim reapers of $\mathcal{G}$ and the minimal surface $v^0$. Let 
 $$\mathcal{S}_f^*=\{u\in \mathcal{S}_f: w_\theta\leq u\leq v^0,\, \mbox{ for every } w_\theta\in\mathcal{G}\}.$$
We point out that   $\mathcal{S}_f^*$ is not empty because $v^0\in  \mathcal{S}_f^*$. By using the maximum principle, it is not difficult to see that    set $\mathcal{S}_f^*$ is stable for the operator $M_D$, that is, if $u\in \mathcal{S}_f^*$, then $M_D[u]\in \mathcal{S}_f^*$. The key point is the next proposition.

\begin{proposition}[Perron process] \label{p-pe}
The function $v:\Omega_m\rightarrow\r$ given by 
$$v(x,y)=\inf\{u(x,y):u\in \mathcal{S}_f^*\}$$
 is a solution of (\ref{eqs}) with $v=\varphi_f$ on $\partial\Omega_m$.\end{proposition}

\begin{proof} The proof consists of two parts.

 Claim 1. {\it The function $v$ is a solution of equation (\ref{eqs})}.

The proof is standard and here we follow \cite{gt}. Let $p\in\Omega_m$ be an arbitrary fixed point of $\Omega_m$.  Consider a sequence $\{u_n\}\subset \mathcal{S}_f^*$ such that $u_n(p)\rightarrow v(p)$ when $n\rightarrow\infty$. Let $D$ be  a closed round disk centered at $p$ and contained in $\Omega_m$. For each $n$, define on $\overline{\Omega_m}$ the function
$$v_n(q)=\min\{u_1(q),\ldots,u_n(q)\},\quad q\in\overline{\Omega_m}.$$
Then $v_n\in \mathcal{S}_f^*$ by Lemma \ref{p-p}. Since $M_D[v_n]\in \mathcal{S}_f^*$,   we deduce  $M_D[v_n](p)\rightarrow v(p)$ as $n\rightarrow\infty$. Set $V_n=M_D[v_n]$. Then $\{V_n\}$ is a decreasing sequence bounded from below by $w_\theta$ for all $w_\theta\in\mathcal{G}$ and satisfying (\ref{eqs}) in the disk $D$. It turns out that the functions $V_n$ are uniformly bounded on compact sets $K$ of $D$. In each compact set $K$, the norms of the gradients $|DV_n|$ are bounded by a constant depending only on $K$ and using H\"{o}lder estimates of Ladyzhenskaya and Ural'tseva, there exist uniform $C^{1,\beta}$ estimates for the sequence $\{V_n\}$ on $K$ (\cite{gjj}). By compactness, there is a subsequence of $V_n$, that we denote $V_n$ again, such that $\{V_n\}$ converges on $K$ to a $C^2$ function $V$ in the $C^2$ topology and by continuity, $V$ satisfies (\ref{eqs}). Moreover, by construction, at the fixed point $p$ we have $V(p)=v(p)$.

It remains to prove that $V=v$ in $\mbox{int}(D)$. For $q\in \mbox{int}(D)$,   the same argument as before gives the existence of   $\{\tilde{u}_n\}\subset\mathcal{S}_f^*$ with $\tilde{u}_n(q)\rightarrow v(q)$. Let $\tilde{v}_n=\min\{V_n,\tilde{u}_n\}$ and $\tilde{V}_n=M_D[\tilde{v}_n]$. Again $\tilde{V}_n$ converges on $D$ in the $C^2$ topology to a $C^2$ function $\tilde{V}$ satisfying (\ref{eqs}) and $\tilde{V}(q)=v(q)$. By construction, $\tilde{V}_n\leq \tilde{v}_n\leq V_n$, hence $\tilde{V}\leq V$.  In view that  $v\leq \tilde{V}$, we infer $\tilde{V}(p)=v(p)=V(p)$. Thus $V$ and $\tilde{V}$ coincide at an interior point of $D$, namely, the point $p$, and both functions $V$ and $\tilde{V}$ satisfy the translating soliton equation. Because $\tilde{V}\leq V$,  the touching principle implies  $V=\tilde{V}$ in $\mbox{int}(D)$. In particular, $V(q)=\tilde{V}(q)=v(q)$. This shows that $V=v$ in $\mbox{int}(D)$ and the claim is proved. 

In order to finish the proof of Theorem \ref{t-strip}, we prove that the function $v$ takes   the value $\varphi_f$ on $\partial\Omega_m$ and consequently, $v$ is continuous up to $\partial\Omega_m$ proving that $v\in C^2(\Omega_m)\cap C^0(\overline{\Omega_m})$. In contrast to with the proof of Theorem  \ref{t-ex}, here we will find local barriers for each boundary point $p\in\partial\Omega_m$.

Claim 2. {\it The function $v$ is continuous up to $\partial\Omega_m$ with $v=\varphi_f$ on $\partial\Omega_m$}.

The graph  of $\varphi_f$ consists of two copies of   $f$,  
$$\Gamma_{\varphi_f}=\Gamma_1\cup\Gamma_2=\{(x,m,f(x)):x\in\r\}\cup \{(x,-m,f(x)):x\in\r\}.$$
Let $p=(x_0,m)\in\partial\Omega_m$ be a boundary point  (similar argument if $p=(x_0,-m)$). Because of the convexity of $f$, in the plane of equation $y=m$ the tangent line $L_p$ to the  planar curve $\Gamma_1$  leaves $\Gamma_1$ above $L_p$. We choose the number $\theta$ such that the grim reaper  $w_\theta$ takes the values $L_p$ on $\partial\Omega_m$: exactly, $\theta$ is chosen so $\tan\theta$ is the slope of $L_p$. Recall that all rulings of this grim reaper are parallel to $L_p$. Let $w_\theta^p=w_\theta$ denote this grim reaper in order to indicate its dependence on the point $p$.

Taking into account the  symmetry of $\varphi_f$ and the convexity of $f$,  we have $w_\theta^p(p)=f(x_0)$ and $w_\theta^p<f$ in $\Gamma_{\varphi_f}\setminus\{(x_0,m,f(x_0)),(x_0,-m,f(x_0))\}$, or in other words, $\partial \Sigma_{w_\theta^p}$ lies strictly below $\partial\Sigma_v$, except   at the points $(x_0,m,f(x_0)$ and $(x_0,-m,f(x_0))$, where both graphs  coincide.

Therefore the function $w_\theta^p$ and the minimal surface $v^0$ form a modulus of continuity in a neighborhood of $p$, namely, $w_\theta^p\leq v\leq v^0$. Because  $w_\theta^p(p)=v^0(p)=f(p)$, we infer that  $v(p)=f(p)$ and this completes the proof of Theorem \ref{t-strip}.
 \end{proof}
 
 We finish this section with the second type of domains, that is, when $\Omega$ is    an unbounded convex domain contained in a strip. Under this situation, we will suppose   $\varphi=0$ on $\partial\Omega$.

\begin{theorem}\label{t-un} 
 Let $\Omega$ be an unbounded convex domain contained in a strip of width strictly less than $\pi$.   Then there is a solution of the translating soliton equation (\ref{eqs}) in $\Omega$ with   $\varphi=0$ on $\partial\Omega$.
\end{theorem}

\begin{proof}
If $\Omega$ is a strip, then the result was established in Theorem   \ref{t-strip}. In fact, if $\Omega=\Omega_m$, $m<\pi/2$, the solution is $w(x,y)=-\log(\cos(y))+\log(\cos(m))$. 

Suppose now that $\Omega$ is not a strip. After a change of coordinates,   we assume that the narrowest strip containing $\Omega$ is $\Omega_m$. Since $\Omega$ is an unbounded domain contained in a strip, then $\partial\Omega$ has two branches asymptotic to the boundary set $\partial\Omega_m$ and the $x$-coordinate function is bounded in $\partial\Omega$ from above or from below.

We follow the same reasoning as in Theorem \ref{t-strip}, and we only point out    the differences. The subsolution is the function $v^0=0$, which is a solution of the minimal surface equation. We consider the family of operators $M_D$  and 
 $$\mathcal{S}=\{u\in C^0(\overline{\Omega}):M_D[u]\leq u \mbox{ for every closed round disk $D\subset\Omega$}, 0\leq u \mbox{ on $\partial\Omega$}\}.$$ 
 Let the grim reaper  $w(x,y)=-\log(\cos(y))$ whose domain is the strip $\Omega^0$ of width $\pi$ and   define $\omega(x,y)=-\log(\cos(y))+\log(\cos(m))$. Note that $\omega=0$ on $\partial\Omega_m$ and  $\omega<0$ on $\partial\Omega$ because $\Omega\subset\Omega_m$. We   construct a solution of equation (\ref{eqs})  between the  grim reaper $\omega$   and the minimal surface $v^0$. Let 
 $$\mathcal{S}^*=\{u\in \mathcal{S}: \omega\leq u\leq 0\}.$$
Note that   $\mathcal{S}^*$ is not empty because $0\in  \mathcal{S}^*$: indeed, $Q[0]=-1<0=Q[\omega]$ and $\omega<v^0$ in $\partial\Omega$, hence $\omega<0$ in $\Omega$ by the comparison principle.  Again, the function  
$$v(x,y)=\inf\{u(x,y):u\in \mathcal{S}^*\}=\inf\{M_D[u](x,y):u\in \mathcal{S}^*\}$$
 is a solution of (\ref{eqs}) and it remains to prove that  the function $v$ is continuous up to $\partial\Omega$ with $v=0$ on $\partial\Omega$. Here the barrier construction in the proof of Theorem \ref{t-strip}   can be adapted to provide boundary modulus of continuity estimates.

Let $p=(x_0,y_0)\in\partial\Omega$ be a boundary point of $\Omega$. We  rotate $\Omega$  with respect to the $z$-axis and      translate along a horizontal direction if necessary, in such way that the tangent line $L$ to $\Omega$ at $p$ is one of the boundaries of $\Omega_m$ and  a neighborhood $U_p$ of $p$ in $\Omega$ is contained in $\Omega_m$: this is possible by the convexity of $\Omega$. There is no loss of generality in assuming  that $L=\{(x,m,0): x\in\r\}$. We take now the restriction of $\omega$, $\omega^*=\omega_{|\Omega_m^*}$,  in the half-strip $\Omega_m^*=\{(x,y)\in\r^2: 0<y\leq m\}$.  Let $\Sigma_{\omega^*}$ be the graph of $\omega^*$.     Notice that $\partial\Sigma_{\omega^*}$ is formed by two parallel lines, one is $L$ and the other one is $L'=\{(x,0,\omega(0)):x\in\r\}$.

Let ${\bf n}(p)=(0,1,0)$  be the unit  outward normal vector to $\partial\Omega$ at $p$. Let  us move horizontally $\Sigma_{\omega^*}$ in the direction ${\bf n}(p)$ until $\Sigma_{\omega^*}$ does not intersect   $\Sigma_v$. Then we come back in  the direction $-{\bf n}(p)$ until the first touching point $q$ between $\Sigma_{\omega^*}$ and $\Sigma_v$. Since $\omega<v<0$ in $\Omega$, it is not possible that  $q\in L'$. By the tangency principle,  $q\in L$ and by the convexity of $\Omega$, the point $q$ coincides with $p$. Accordingly, we have proved that in the interior of the neighborhood $U$, we have $\omega<v<0$. Since  $\omega(p)=v(p)=0$, the functions $\omega$ and $0$ are a modulus of continuity in $U$ of $p$, hence $v(p)=0$. This completes the proof of Theorem \ref{t-un}.

\end{proof}
 
 We point out that   the domain $\Omega$ is not necessarily strictly convex. Thus in the last part of the above proof, the intersection between $\Sigma_{\omega^*}$ and $\Sigma_v$ at the first touching point, may occur along a segment of $L$. In any case, we can take that a first contact point is the very point $p$. 

\end{document}